\documentclass[a4paper,12pt]{article}
\usepackage{bm}
\usepackage{amsmath}
\usepackage{amsthm}
\usepackage{wasysym}
\usepackage{amssymb}
\usepackage{amsfonts}
\usepackage{graphicx}
\usepackage[english]{babel}
\usepackage[utf8]{inputenc}%a la place d'applemac utf8
\usepackage[T1]{fontenc}
\usepackage{mathrsfs}
\usepackage{enumerate}
\usepackage{version}
\usepackage{calc}
\usepackage{subfigure}
\usepackage{authblk}
\usepackage{comment}
 \usepackage{pstricks,pst-math,pst-xkey}
 \usepackage{float}
 \usepackage{indentfirst}%%%%%%%%%%%%%%%%%%%%%%%%%%
\usepackage[pagewise]{lineno} %\linenumbers
\newtheorem{definition}{Definition}[section]
\newtheorem{theorem}[definition]{Theorem}
\newtheorem{lemma}[definition]{Lemma}
\newtheorem{corollary}[definition]{Corollary}
\newtheorem{proposition}[definition]{Proposition}
\newtheorem{remark}[definition]{Remark}

\numberwithin{equation}{section}

\newcommand{\N}{{\mathbb N}}

\newcommand{\R}{{\mathbb R}}

\newcommand{\ve}{{\varepsilon}}

\def\ve{\varepsilon}

\def\f{\varphi}

\def\d{\partial}

\def\tu{\tilde{u}}

\def\N{\mathbb{N}}

\def\T{\mathbb{T}_L}

\def\f{\varphi}

\def\R{\mathbb{R}}
\def\E{\mathbb{E}}
\def\N{\mathbb{N}}

\begin{document}

\title{Long Time Behavior of Stochastic Thin Film Equation}

\author[1]{Oleksiy Kapustyan \thanks{kapustyan@knu.edu}}

\author[3] {Olha Martynyuk
\thanks{o.martynyuk@chnu.edu.ua}}

\author[2]{Oleksandr Misiats\thanks{omisiats@vcu.edu}}

\author[1]{Oleksandr Stanzhytskyi \thanks{stanzhytskyi@knu.edu}}

\affil[1]{Department of Mathematics,
Taras Shevchenko National University of Kyiv, Ukraine}

\affil[2]{Department of Mathematics, Virginia Commonwealth University,
Richmond, VA, 23284, USA}

\affil[3]{Department of Mathematics, Yuriy Fedkovych Chernivtsi National University, Chernivtsi, Ukraine}

%\author{Oleksandr~\textsc{Misiats}\\
% Department of Mathematics, Purdue University\\
%West Lafayette, IN, 47907, USA
%\\{\tt omisiats@purdue.edu}
%\and Oleksandr~\textsc{Stanzhytskyi}\\
%Department of Mathematics, Kiev National University, Kiev, Ukraine
%\\{\tt ostanzh@gmail.com}
%\and Nung Kwan~\textsc{Yip}\\
% Department of Mathematics, Purdue University\\
%West Lafayette, IN, 47907, USA
%\\{\tt yip@math.purdue.edu}
%}

%\and {Oleksandr~\textsc{Stanzhytskyi}
%\\Department of Mathematics, Kiev Nantional University, Kiev, Ukraine
%\\{\tt misiats@math.psu.edu}}

\maketitle

\begin{abstract}
We consider the stochastic thin-film equation with linear deterministic and stochastic It\^o perturbations. The existence of nonnegative weak
martingale solutions on the semi-axis is established, and their asymptotic
behavior as $t \to \infty$ is investigated. It is shown that in square mean the $L^\infty$ norm of the solution converges to the spatial mean value of the initial condition, multiplied by a random factor similar to a geometric Wiener process.
\end{abstract}

\section{Introduction}

We consider the stochastic thin-film equation
\begin{equation}
\mathrm{d}u
=
- \partial_x\!\left(u^n \partial_x^3 u\right)\,\mathrm{d}t
+ \gamma(t) u \,\mathrm{d}t
+ \alpha(t) u \,\mathrm{d} \beta (t),
\label{1.1}
\end{equation}
for $t \geq 0$, $x \in [0,L]$, where $L>0$ and $\mathbb{T}_L$ denotes the
one-dimensional torus of length $L$. Denote $Q_T:=[0,T] \times [0,L]$.
We will always assume periodic boundary conditions
\[
\partial_x^i u(t,0) = \partial_x^i u(t,L), \qquad  t \geq 0, 
\qquad i=0,1,2,3,
\]
and nonnegative initial data
\[
u(0,x) = u_0(x) : \mathbb{T}_L \to [0,\infty).
\]
The process $\beta(t)$ is a Wiener process defined on a complete filtered probability
space
\[
(\Omega,\mathcal{F},(\mathcal{F}_t)_{t\ge 0},\mathbb{P}), \ t \geq 0,
\]
with a complete and right-continuous filtration $(\mathcal{F}_t)_{t\ge 0}$.
The functions $\gamma(t)$ and $\alpha(t)$ describing the deterministic and stochastic
perturbations, respectively, are continuous for $t \geq 0$. The main result of this work is to establish the asymptotic behavior of solutions $u(t, x, \omega)$ as $t \to \infty$, in particular, to show that, under suitable assumptions, for $n \geq 4$ the solutions either converge to zero, or
\begin{equation*}
\mathbb{E}\,\|u(t,\cdot)-\bar u_0 \eta(t)\|_{L^\infty}^2 \to 0,
\qquad t \to \infty,
\label{1.2}
\end{equation*}
where $$\bar u_0 = \frac{1}{L} \int_0^L u_0(x) \, dx,$$
and $\eta(t)$ is a random process of geometric Wiener process type
\begin{equation*}
\eta(t) = \exp\!\left(
\int_0^t (\gamma(s) - 
\frac12 \alpha^2(s) ) \, ds + \int_0^t \alpha^2(s)\, d\beta(s)
\right).
\end{equation*}
The process $\eta(t)$ is a Gaussian process with independent increments, and for any $p \in \R$ we get 
\begin{equation}
\E |\eta(t)|^p = \exp\!\left(
p \int_0^t (\gamma(s) - 
\frac12 \alpha^2(s) ) \, ds + \frac{p^2}{2} \int_0^t \alpha^2(s)\, d\beta(s)
\right).
\label{ast}
\end{equation}

The equations of type  \eqref{1.1}  model the motion of liquid droplets of thickness $u$, spreading over the solid surface. It can be derived from the lubrication theory under the so-called ``thin-film assumption'', that is, that the dimensions in the horizontal directions are significantly larger than in the vertical (normal) one. Under this assumption, the surface tension plays the key role in the dynamics of the droplet. The regions where  $u>0$ are called the {\it wetted regions}. The equation is parabolic inside these regions, and degenerate on the boundary of these regions. This way one may think of  \eqref{1.1} as of a  fourth-order nonlinear free boundary problem inside a wetted region, which evolves in time  with finite propagation speed \cite{Bernis}. 

The mathematical analysis of equations of type \eqref{1.1} is particularly tricky due to the lack of comparison (maximum) principle, which is one of the primary analytical tools in classic parabolic theory. One of the pioneering works in this direction was the paper  \cite{Bern} by Bernis and Friedman. In it, the authors established the existence of a non-negative generalized weak solution using the energy-entropy method. This solution was obtained as a limit of the solutions of the regularized problems. It is worth noting that in this work, the authors introduce the notion of a weak solution via an integral identity inside wetted regions. This notion is somewhat ``weaker'' than the subsequent definitions, which involve the integral identities over the entire region $\mathbb{T}$.  In \cite{Bern} the authors also describe the behavior of the support of the solution. 

 The work \cite{DalPas} addressed existence of more regular (strong or entropy) solutions.  In this work the authors also studied the asymptotic in $t$ behavior of the solution. In particular, it was shown that under certain assumptions on the initial conditions,
 \begin{equation}\label{1.3}
 \sup_{x \in L}\left|u(t,x) - \frac{1}{L} \int_0^L u_0(x) \, dx \right| \to 0, \ t \to \infty.
 \end{equation}
 A similar result was obtained in \cite{BertPugh} using a somewhat different approach using entropy estimates. Furthermore, this work established that the rate of the aforementioned convergence is exponential.  The work \cite{Tud} makes use of the energy estimates for the regularized problem to obtain a stronger result, that is, the exponential decay of the energy
 \begin{equation}\label{energy}
 J[u]:= \frac{1}{2} \int_0^L u_x^2(t,x) \, dx.
 \end{equation}
 
The work \cite{KapTar} established the existence of a global trajectory attractor for the generalized thin film equation with a nonlinear dissipative term, which describes the relation between nonlinear absorption and spatial injection. The result \cite{Wang} extends the analysis of a related  stochastic PDE to the homogenization setting.  In  \cite{Grun95} the problem of such type appeared in modeling the evolution of the dislocation density in the theory of plasticity. The work \cite{Jin} considered the thin film equation with nonlocal elliptic negative operator and with inhomogeneous forces relevant to hydraulic fracture modeling.

In this paper we consider the thin film equation with stochastic perturbation. It is worth mentioning that one needs to be very careful when dealing with the effect of noise on nonlocal and/or ill-posed problems. In some cases, e.g. \cite{MisStaTop}, the presence of a even small stochastic perturbation leads to a finite time blowup while an unperturbed equation has global solution. In others \cite{Sta}, the effect of the random perturbation is exactly opposite - it may lead to the existence of a global solution while the corresponding deterministic equation has a finite time blowup. The long time behavior of stochastically perturbed evolution equations is typically described via the existence and properties of invariant measures, see, e.g. \cite{MisStaYip1}, \cite{MisStaYip2}, \cite{MisStaYip3}, \cite{MisStaSta}, \cite{MisStaHie}, \cite{MisStaKap}, \cite{MisStaCla}.

The stochastic version of thin-film equation was first introduced in \cite{DavMorSlo}, it described  modeling the enhanced spreading of droplets.  In the subsequent work \cite{GruMeRa} the authors additionally take the interface potential between fluid and substrate into account. This prevents the solution $u$ from becoming negative, and allows to describe the coarsening and de-wetting phenomena. 

The first rigorous construction of a non-negative martingale solution of the stochastic thin film equation with Ito noise and additional interface potential was obtained in \cite{FisGru}. This result was derived by constructing a spatially discrete approximation of the solution.  In \cite{Cor}, the author considered a more general case of the main operator in the form $-\partial_x(u^{n}u_{xxx})$ (referred as a more general {\it{mobility}}), and established the conditions for the existence of a global strong solution for this problem. 

The work \cite{Gess1}  established the existence of a nonnegative martingale solution for \eqref{1.1} with quadratic mobility.  The main tool in \cite{Gess1} was Trotter-Kato type decomposition of the dynamics into deterministic and stochastic parts, which are eventually coupled. The subsequent work \cite{Gess2} established the similar result for
\[
du  = - \partial_x(u^n \partial_x^3 u)\, dt + \partial _x (u^{\frac{n}{2}} \circ  dW)
\]
with $n \in [\frac{8}{3}, 4)$.  To this end, the authors started with regularizing the problem by replacing the mobility $u^n$ with $u^n \to (u^2+\ve^2)^{n/4}, \ \ve>0$, which makes the problem non-degenerate. Then, by means of Galiorkin approximations, the authors  establish the existence of a weak solution for the regularized equation, and  pass to the limit $\ve \to 0$.  The authors also mention that the Trotter-Kato scheme is not applicable in this case. In particular, in the case $n = 3$, when estimating the term $\E \|\d_x u\|_2^2$ the term of the form $\E \int_0^L u^{-1} (\d_x u)^4 \, dx$ emerges and thus hinders us from obtaining an energy estimate. Such estimate, however, may be established using the entropy estimates for the full equation. The use of $\alpha$ - entropy estimates in the work \cite{Sauer} allowed to generalize the results of \cite{Gess2} to the case of $n \in (2,3)$.

In \cite{KapMisSta23} we established the existence of a non-negative martingale solution of the nonlinear stochastic thin-film  equation with nonlinear deterministic and stochastic drift coefficients
\begin{equation*}
du = (-\partial_x(u^{2}u_{xxx})  + l(u))dt + \partial_x (u \circ dW)+ f(u)dW_1(t),
\end{equation*}
where $W$ and $W_1$ are independent $Q$-Wiener processes. 
The above equation has both nonlinear absorption $l(u)$, as well as the nonlinear stochastic Ito perturbation $f(u) d W_1(t)$. Due to the presence of these two terms, the equation is no longer in divergence form, which makes its analysis significantly different from the one in \cite{Gess1} and \cite{Gess2}. 

The literature on the asymptotic long-time behavior of solutions of a stochastic thin-film equation, on contrary, is pretty sparce. To the best of our knowledge, we are aware only of the recent work \cite{GRLO}, which studied the finite speed of propagation of the solution.

Let us outline the strategy of obtaining the main result on the long-time behavior of solutions to equation (1.1).

\medskip

\noindent
\textbf{Step 1.}
First, using the Trotter--Kato scheme, we establish the existence of a weak solution on a finite interval $[0,T]$.

\medskip

\noindent
\textbf{Step 2.}
Prove the the local solution from Step 1 is in fact global-in-time solution by means of Kolmogorov Theorem.

\medskip

\noindent
\textbf{Step 3.}
Using entropy estimates, we establish the positivity of the solution.

\medskip

\noindent
\textbf{Step 4.} Show that the energy (and, in certain cases, the mass) converge to zero. The main result in this case will follow using either the analysis of a certain linear equation, or from Poincare inequality.

\medskip

The paper is structured as follows. In Section~\ref{Sec2} we introduce the notation, definitions, and state preliminary results as well as the main result. In Section \ref{Sec 3} through Section \ref{Sec 5} we derive the existence of a positive martingale solution of equation \eqref{1.1} on $t \geq 0$. Finally, in Section \ref{Sec 6} we establish the main result. 

\section{Preliminaries and Main Results} \label{Sec2}

Throughout the paper, we will be using the following notation. For $u,v \in L^2(\mathbb{T}_L)$, let
\[
(u,v)_2 := \int_{L} u(x)\,v(x)\,dx,
\qquad
\|u\|_2 := \sqrt{(u,u)_2},
\]
where $\int_{L} u(x)\,dx$ stands for  $\int_0^{L} u(x)\,dx$.
Next, for $Q \in \R^d$ with $\partial Q \in C^{\infty}$, for $s \in [0, \infty)$, $p \in [1, \infty)$, let $W^{s,p}(Q)$ be the regular Sobolev space for $s \in \N$, and Sobolev-Slobodeckij space for non-integer $s$. For $p=2$ we will denote $W^{s,p}(Q) = H^s(Q) = H^s$. If $X$ is a Banach space, the space $C^{k+\alpha}(Q;X)$ is the space of $k$ times differentiable functions $Q \to X$, whose $k$-th derivatives are Holder-continous with exponent $\alpha \in (0,1)$ on compact subsets of $Q$. We also denote $C^{k-}(Q;X)$ to be the space of $k-1$-times differentiable functions, whose $k-1$-st derivatives are Lipschitz continuous. The space $BC^{0}(Q,X)$ is the set of bounded continuous functions.  We also denote $H^1_w(\mathbb{T}_L)$ to be the space of  $H^1(\mathbb{T}_L)$ functions endowed with the weak topology, induced by $\| \cdot \|_{1,2}$. 
Finally, for  $u: Q_T \to \R$ we denote 
\[
P_T := \{(t,x) \in \bar{Q}_T, u(t,x)>0\}.
\]

\begin{definition}\label{Def:2.1}
A {\bf weak martingale solution} of \eqref{1.1} on $[0,T]$ with
initial data $u_0 \in H^1(\mathbb{T}_L)$,
is a triple
\[
\{(\tilde{\Omega},\tilde{\mathcal{F}}, \tilde{\mathcal{F}}_t,\tilde{\mathbb{P}}), \tilde{\beta}, \tilde{u}\}
\]
such that $(\tilde{\Omega},\tilde{\mathcal{F}}, \tilde{\mathcal{F}}_t,\tilde{\mathbb{P}})$ is a filtered probability space,
$ \tilde{\beta}$ is a real--valued $(\tilde{\mathcal{F}}_t)$--Wiener process, and
$\tilde{u}$ is an $(\tilde{\mathcal{F}}_t)$--adapted continuous $H^1(\mathbb{T}_L)$--valued process
satisfying:
\begin{enumerate}
\item
\[
\mathbb{E}\!\left[\sup_{t\in[0,T]}\|\tilde{u}(t)\|_{H^1(\mathbb{T}_L)}^2\right] < \infty;
\]

\item
For almost all $(\omega,t) \in \Omega \times [0,T]$, the weak third--order derivative
$\partial_x^3 \tilde{u}$ exists on $\{\tilde{u}(t) \neq 0\}$ and satisfies
\[
\tilde{\mathbb{E}}\!\left(\int_0^T \!\!\int_{\{\tilde{u}(t) \neq 0\}}
\tilde{u}(t,x)^{2n} \bigl|\partial_x^3 \tilde{u}(t,x)\bigr|^2 \, dx \, dt \right)^{1/2}< \infty;
\]

\item
For any $\varphi \in C^\infty(\mathbb{T})$, $\tilde{\mathbb{P}}$--almost surely,
for all $t \in [0,T]$,
\begin{align}
(\tilde{u}(t),\varphi)_2
&= (u_0,\varphi)_2
+ \int_0^t \int_{\{\tilde{u}(s)>0\}}
\tilde{u}^n(s) \partial_x^3 \tilde{u}(s)\,\partial_x \varphi \, dx\,ds \notag \\
&\quad
+ \int_0^t \gamma(s)(\tilde{u}(s),\varphi)_2\,ds
+ \int_0^t \alpha(s)(\tilde{u}(s),\varphi)_2\,d \tilde{\beta}(s).
\label{2.1}
\end{align}
\end{enumerate}
\end{definition}

\medskip

The following theorem establishes the existence of a nonnegative solution of equation \eqref{1.1}
in the sense of Definition \ref{Def:2.1}.

\begin{theorem}\label{Th:2.1}
Assume $n > 0$. Then for any
$u_0 \in H^1(\mathbb{T})$, $u_0 \ge 0$, $u_0 \not\equiv 0$,
equation \eqref{1.1} admits a nonnegative weak martingale solution $\tilde{u}$, which is defined on $t \in [0,T]$ for any $T > 0$, and for any $p \geq 2$ satisfies
\[
\tilde{\E} \sup_{t \in [0,T]} \|\tilde{u}(t,\cdot)\|_2^p \leq C \| u_0\|^p_2
\]
for some $C>0$ independent of $\tilde{u}$ and $u_0$.
\end{theorem}

The next result guarantees that the local solution, estalished in Theorem \ref{Th:2.1}, is in fact global.

\begin{theorem}\label{th:2.2}
Assume $n>2$ and $u_0 \in H^1(\mathbb{T}_L)$, is such that $u_0 \ge 0$, and
\[
\int_L u_0^{2-n}(x) \, dx < \infty.
\]
Then the equation~\eqref{1.1} admits a nonnegative martingale solution
$\tu(t)$ on $[0,\infty)$ such that for any $t>0$,
the set
\[
\{x \in L : \tu(t,x,\omega)=0\}
\]
has zero Lebesgue measure, $\tilde{\mathbb{P}}$--almost surely.
\end{theorem}

\medskip

Setting $\varphi \equiv 1$ in \eqref{2.1}, we obtain for almost every $\omega$
the following stochastic equation for the mass:
\begin{equation}
\int_{L} \tu(t,x)\,dx
=
\int_{L} u_0(x)\,dx
+
\int_0^t \gamma(s) \int_{L} \tu(s,x)\,dx\,ds
+
\int_0^t \alpha(s) \int_{L} \tu(s,x)\,dx\,d \tilde{\beta}(s).
\label{2.2}
\end{equation}
Solving the linear equation \eqref{2.2} yields
\begin{equation*}
\int_{L} \tu(t,x)\,dx = \eta(t) \int_{L} u_0(x)\,dx,
\label{2.3}
\end{equation*}
where $\eta(t)$ is defined in \eqref{1.3}.

The main results concerning the asymptotic behavior of the solutions as $t \to \infty$ are contained in the theorems below.

\begin{theorem} \label{Th:2.3}[Large perturbations] 
Assume $n \ge 4$. 
\begin{itemize}
    \item 
If
\begin{equation}
\int_0^\infty \alpha^2(t) \, dt < \infty, \text { while }
\qquad
\int_0^\infty \gamma(t) \, dt = \infty,
\label{2.4}
\end{equation}
then
\[
\|u(t,\cdot, \omega)\|_{L^\infty} \to 0
\quad \text{as } t \to \infty
\]
both in square mean and with probability 1. 
\item If 
\begin{equation}\label{2.6}
\int_0^\infty \alpha^2(t) \, dt < \infty, \text{ and } \limsup_{t \to \infty} \frac{\int_0^t [\gamma(s) - \frac{\alpha^2(s)}{2}] \, ds}{[2 \int_0^t \alpha^2(s) ds \ln \ln \int_0^t 4 \alpha^2(s) ds]^{1/2}} < -1
\end{equation}
then
\[
\|u(t,\cdot, \omega)\|_{L^\infty} \to 0
\quad \text{as } t \to \infty
\]
with probability 1. 
\end{itemize}
\end{theorem}

\begin{remark}
    The second condition in \eqref{2.6} can be satisfied even for positive $\gamma(t)$, for example, if 
    \[
\gamma(t) \leq (\frac12 - \ve)   \alpha^2(t), 0 \leq \ve \leq \frac12. 
\]
Physically this means that even if the mass is being added to the film externally, the  presence of a high intensity noise (first condition in \eqref{2.6}) still ruins the film. 
\end{remark}

\begin{theorem}\label{Th:2.5}[Bounded perturbations]
Assume $n \ge 4$. If 

\begin{equation}
\int_0^\infty \alpha^2(t) \, dt < \infty, \text { and }
\qquad
\int_0^\infty \gamma(t) \, dt < \infty,
\label{2.7}
\end{equation}
as well as 
\begin{equation}\label{2.7ast}
2 \gamma(t) + \alpha^2(t) \to 0, \ t \to \infty
\end{equation}
then
\begin{equation}
\mathbb{E}\|u(t,\cdot) - \eta(t) \bar{u}_0\|_{L^\infty}^2
\to 0
\quad \text{as } t \to \infty,
\label{2.8}
\end{equation}
and $\eta(t)$ is defined in \eqref{1.3}.

\end{theorem}

\begin{remark}\label{rem:1}
    The convergence \eqref{2.8} differs from the classic convergence to the mean, i.e. \eqref{1.3}, due to the presence of the weight $\eta(t)$. However, we may obtain the convergence result  \eqref{1.3} with probability 1 for the solution $u(t,x)$ of the stochastic thin film  equation with quadratic mobility, which has the form
\begin{equation}\label{1.2*}
du(t,x) = \left(\partial_x(-u(t,x)^{2}u_{xxx}(t,x)) + \frac{1}{2}  \partial_{xx}  u(t,x) \right)dt +
\partial_x u(t,x) \, d \beta(t).
\end{equation}
The a.e. exponential decay of the energy \eqref{energy} also holds for the solutions of \eqref{1.2*}.
\end{remark}

\section{Local Existence of Solutions (Step 1) } \label{Sec 3}

In this section we prove Theorem \ref{Th:2.1}. The strategy of the proof is to apply a Trotter--Kato type decomposition of the dynamics into deterministic and stochastic parts. This method was used by a number of authors, e.g.\ \cite{Gess1,FisGru},
in particular to establish the existence of SPDEs with locally Lipschitz coefficients. For the convenience of the reader, we split the proof into the
deterministic and stochastic parts, outlined below.\\

{\bf Deterministic dynamics.} Consider the following deterministic equation:
\begin{equation}
\begin{cases}
\partial_t v = -\partial_x \bigl( v^n \partial_x^3 v \bigr), \\
v(0) = u_0.
\end{cases}
\label{3.1}
\end{equation}

The following results follows from \cite{DalPas}, Proposition 1.1.

\begin{theorem}[Existence of weak solutions]\label{Th:3.2}
Assume $u_0 \in H^1(\mathbb{T}_L)$, $u_0 \ge 0$, and $u_0 \not\equiv 0$.
Then the equation \eqref{3.1} has a solution 
$v(t,x): [0,T) \times \mathbb{T}_L \to [0, \infty)$, which satisfies
\begin{equation*}
\int_0^T \!\!\int_{L} v \,\partial_t \varphi \, dx\,dt
+ \int_0^T \!\!\int_{v(t, \cdot)>0}
v^n \partial_x^3 v \,\partial_x \varphi \, dx\,dt
= 0
\end{equation*}
for all $\varphi \in C_c^\infty((0,T); C^{\infty}(\mathbb{T}_L))$.
This solution $u$ satisfies the following properties:
\begin{enumerate}
\item $v \in C^{\frac18,\frac12}\!([0,T] \times \mathbb{T}_L)$;
\item $v(0)=u_0$ in the sense that
\[
\|v(t)-u_0\|_{1,2} \to 0 \quad \text{as } t\to 0; 
\]
\item $v \in L^\infty\!\left([0,T]; H^1(\mathbb{T}_L)\right)$;
\item $v^n \partial_x^3 v \in L^2(\{v>0\})$;
\item Mass conservation:
\[
\int_{L} v(t,x)\,dx
=
\int_{L} u_0(x)\,dx,
\qquad t\in[0,T].
\]
\end{enumerate}
\end{theorem}
\medskip

The solution is obtained as a uniform limit $v_\varepsilon \to v$  in $\overline{Q_T}$, where $v_\ve$  is the unique positive smooth solution of
\begin{equation*}
\begin{cases}
\partial_t v_\varepsilon
= -\partial_x\!\left(
f_\varepsilon(v_\varepsilon) \partial_x^3 v_\varepsilon
\right), \\[0.3em]
v_\varepsilon(0,x) = u_{0,\varepsilon}(x),
\end{cases}
\label{3.2}
\end{equation*}
where $u_{0,\varepsilon}\in C^\infty(\mathbb{T}_L)$,
$u_{0,\varepsilon}>0$, and
$u_{0,\varepsilon}\to u_0$ in $H^1(\mathbb{T}_L)$ as $\varepsilon\to 0$. Moreover, the following a priori estimate holds:
\begin{equation*}
\|\partial_x v_\varepsilon(t)\|_{L^2}^2
+ 2\int_0^t \!\!\int_{L}
f_\varepsilon(v_\varepsilon)
|\partial_x^3 v_\varepsilon|^2
\,dx\,ds
=
\|\partial_x u_{0\varepsilon}\|_{L^2}^2,
\end{equation*}
for all $t\in[0,T]$, where
\begin{equation*}
f_\varepsilon(s)
=
\frac{s^{n+4}}{\varepsilon s^{n} + s^4}.
\label{eq:feps}
\end{equation*}

%%%%%%%%%%%%%%%%%%%%%%%%%%%%%%%%%%%%%%%%

% The argument continues on the next page.

Following Corollary 2.2 \cite{Gess1}, we have

\begin{corollary}
\label{cor:estimate}
Under the conditions of Theorem \ref{Th:3.2}, the solution of problem \eqref{3.1} satisfies

\begin{equation}
\label{3.4*}
\|\d_x v(t, \cdot)\|_{L^2}^p
+ 2 \int_0^t \|\d_x v(t, \cdot)\|_{L^2}^{p-2}
\int_{\{v(s, \cdot)>0\}} v^n(s,x) |\d_x^3 v(s,x)|^2\, dx \, ds
\le
\|\d_x v_0\|_{L^2}^p
\end{equation}
for all $t \in [0,T]$ and $p \ge 2$.
\end{corollary}

{\bf Stochastic dynamics.} Let $(\Omega,\mathcal{F},(\mathcal{F}_t)_{t\in [0,T]},\mathbb{P})$
be a filtered probability space, where $(\mathcal{F}_t)_{t\in[0,T]}$
is a complete right--continuous filtration, and let $\beta(t)$ be a standard
real--valued $\mathcal{F}_t$--Wiener process. For all $x \in [0,L]$ consider the following linear stochastic equation
on $[0,T]$:
\begin{equation}
\label{3.4}
dw = \gamma(t) w \, dt + \alpha(t) w \, d\beta(t),
\end{equation}
with
\[
w(t_0, x) \in L^p(\Omega,\mathcal{F}_0,\mathbb{P}; H^1(\mathbb{T}_L)).
\]
This equation has a unique solution for $t \geq t_0$, given by
\begin{equation}
\label{3.5}
w(t,x)
=
w(t_0,x)
\exp\left\{
\int_{t_0}^t \left( \gamma(s) - \tfrac12 \alpha^2(s) \right) ds
+ \int_{t_0}^t \alpha(s) \, d\beta(s)
\right\},
\end{equation}

whose regularity in $x$ is solely determined with the regularity of $w(t_0,x)$. 
\begin{lemma}
\label{lem:3.4}
Let $p \in [2,\infty)$ and
\[
w(t_0,x): = w_0 \in L^p(\Omega,\mathcal{F}_0,\mathbb{P}; H^1(\mathbb{T}_L)).
\]
Then for all $0 \leq t_0 \leq t \leq T$ the following properties hold:
\begin{enumerate}
\item
\begin{equation}
\label{3.6}
\mathbb{E} \sup_{s \in [t_0,t]} \|w(s, \cdot)\|_{1,2}^p
\le e^{C_1(t-t_0)} \mathbb{E} \|w (t_0, \cdot)\|_{1,2}^p.
\end{equation}

\item
If $w_0 \ge 0$ almost surely, then
$w(t,x) \ge 0$ almost surely for all $t \in [0,T]$.

\item 
\begin{equation}
\label{3.7}
\mathbb{E}\|\d_x w(t, \cdot)\|_{L^2}^p
\le
\mathbb{E}\|\d_x w(t_0, \cdot)\|_{L^2}^p \,
e^{C_2 (t-t_0)}.
\end{equation}

\item
\begin{equation}
\label{3.7.1}
\mathbb{E}\left|\int_{L} w(t,x)\,dx\right|^p
\le
\mathbb{E}\left|\int_{L} w(t_0,x)\,dx\right|^p
e^{C_1 (t-t_0)}.
\end{equation}
\end{enumerate}
Here $C_1>0$ is a constant independent of $w$.
\end{lemma}

\medskip

\begin{proof}
We use the explicit formula~\eqref{3.5}: 
\begin{equation}
\label{3.8}
w(t,x) = w(t_0,x)\,\eta(t,t_0),
\end{equation}
where
\begin{equation*}
\label{eq:eta-def}
\eta(t,t_0)
=
\exp\left\{
\int_{t_0}^t \left( \gamma(s) - \tfrac12 \alpha^2(s) \right) ds
+ \int_{t_0}^t \alpha(s) \, d\beta(s)\right\}.
\end{equation*}
Since $w(t_0,\cdot)$ is $\mathcal{F}_{t_0}$--measurable and
$\eta(t,t_0)$ is independent of $\mathcal{F}_{t_0}$, we obtain
\begin{equation*}
\label{3.9}
\mathbb{E} \sup_{s \in [t_0,t]} \|w(s, \cdot)\|_{1,2}^p
=
\mathbb{E}\|w(t_0, \cdot)\|_{1,2}^p
\,\mathbb{E} \sup_{s \in [t_0,t]}|\eta(s,t_0)|^p .
\end{equation*}
Furthermore, $\eta(t)$ in its integral form, satisfies
\[
\eta(t, t_0)  = 1 + \int_{t_0}^t \gamma(s) \eta(s) \, ds + \int_{t_0}^t \alpha(s)  \eta(s) d \beta(s).
\]
Therefore, using Burkhover-Davis-Gundy standard estimates for stochastic exponentials, the continuity of
$\gamma$ and $\alpha$, and the martingale properties of the stochastic
integral, we obtain
\[
\mathbb{E} \sup_{s \in [t_0,t]}|\eta(s,t_0)|^p \leq 1 + C \int_{t_0}^t \mathbb{E} \sup_{\tau \in [t_0,s]}|\eta(\tau,t_0)|^p \, ds,
\]
where the constant $C>0$ is independent of $t_0$ and $t$, but depends on $p$ and $T$. The conclusion \eqref{3.6} now follows from Gronwall's inequality. Furthermore, using \eqref{ast} and continuity of $\gamma$ and $\alpha$, we have
\begin{equation}\label{3.10}
  \mathbb{E}\|\d_x w(t, \cdot)\|_{L^2}^p =\mathbb{E}\|\d_x w(t_0, \cdot)\|_{L^2}^p \mathbb{E} |\eta(t,t_0)|^p 
\le
\mathbb{E}\|\d_x w(t_0, \cdot)\|_{L^2}^p \,
e^{C_2 (t-t_0)},
\end{equation}
which implies \eqref{3.7}. The proof of \eqref{3.7.1} is analogous.

\end{proof}
 We now proceed with Trotter-Kato splitting scheme. To this end, for fixed $N \geq 1$ we equi-partition the interval $[0,T]$ into the intervals of length $\delta = \frac{T}{N+1}$. Then for $t \in [(j-1) \delta, j \delta]$ and arbitrary $\varphi \in C^\infty(\mathbb{T}_L)$ define:

\medskip

\paragraph{Deterministic dynamics (D).}
We look for the function $v_N$ which, for
$t \in [(j-1)\delta,j\delta]$, satisfies
\begin{equation*}
\label{3:10}
\langle v_N(t,\cdot),\varphi\rangle_{L^2}
-
\langle v_N((j-1)\delta,\cdot),\varphi\rangle_{L^2}
=
\int_{(j-1)\delta}^{t}
\int_{\{v_N(s,\cdot)>0\}}
v_N^n (\d_x^3 v_N^n) (\d_x \f) \,dx\,ds,
\end{equation*}
for $j=1,\dots,N$.

\medskip

\paragraph{Stochastic dynamics (S).}
We look for the function $w_N$ satisfying
\begin{equation}
\label{3.11}
w_N(t,\cdot) - w_N((j-1)\delta,\cdot)
=
\int_{(j-1)\delta}^{t} \gamma(s) w_N(s,\cdot)\,ds
+
\int_{(j-1)\delta}^{t} \alpha(s) w_N(s,\cdot)\,d\beta(s)
\end{equation}
for $j=1,\dots,N$.
\medskip

\paragraph{Deterministic--stochastic connection (DS).}
We set
\begin{equation}
\label{3.12}
v_N(0):=u_0,
\qquad
v_N(j\delta, \cdot):=\lim_{t\to j\delta-} w_N(t, \cdot),
 \text{  and  }
w_N((j-1)\delta, \cdot):=\lim_{t\to j\delta} v_N(t, \cdot)
\quad \text{a.s.}
\end{equation}
As we showed above, the equations \eqref{3.10} and \eqref{3.11} are well posed. Following the same approach as in \cite{Gess1}, we define the concatenated approximate solution
\[
u_N : [0,T]\times\mathbb{T}_L \times\Omega \to [0,\infty)
\]
via
\begin{equation*}
\label{eq:approx-solution}
u_N(t,\cdot)
=
\begin{cases}
v_N(2t-(j-1)\delta,\cdot),
& t \in [(j-1)\delta,(j-\tfrac12)\delta], \\[0.5em]
w_N(2t-j\delta,\cdot),
& t \in [(j-\tfrac12)\delta,j\delta],
\end{cases}
\qquad j=1,\dots,N+1.
\end{equation*}
This way, $u_N(t,\cdot)$ is well defined and satisfies
$u_N(t,\cdot)\ge0$ for every $t\in[0,T]$.

\medskip

\begin{proposition}
\label{prop:3.5}
For any $p\in[2,\infty)$ there exists a constant $C>0$ such that,
for any $N\ge1$, we have
\[
w_N, v_N, u_N \in L^p\!\left(
\Omega,\mathcal{F},\mathbb{P};
L^\infty(0,T;H^1(\mathbb{T}_L))
\right),
\]
satisfying

\begin{align}
\label{3.13}
& \mathbb{E} \sup_{t \in [0,T]} \|u_N(t,\cdot)\|_{1,2}^p
 + \mathbb{E} \sup_{t \in [0,T]} \|w_N(t)\|_{1,2}^p
 + \mathbb{E} \sup_{t \in [0,T]} \|v_N(t)\|_{1,2}^p \nonumber \\
& \quad
 + \mathbb{E} \int_0^T \|\partial_x v_N(t)\|_{1,2}^{p-2} \int_{\{v_N(t, \cdot)>0\}} v_N^n (\d_x^3 v_N)^2 \, dx \, dt 
 \le C \, \mathbb{E} \|u_0\|_{1,2}^p .
\end{align}
\end{proposition}
\medskip

\begin{proof}  By construction, $u_N$ is continuous in $t$.
Next, we proceed with the estimates of $\partial_x v_N(j \delta)$ and
$\partial_x w_N(j \delta)$.
Using \eqref{3.12}, \eqref{3.7}, and \eqref{3.4}, we have
\begin{align*}
\mathbb{E} \|\partial_x v_N(\delta, \cdot)\|_{L^2}^p
&= \mathbb{E} \left( \lim_{t \to \delta^-}
      \|\partial_x w_N(t, \cdot)\|_{L^2} \right)^p
\le e^{C_2 \delta} \mathbb{E} \|\partial_x w_N(0, \cdot)\|_{L^2}^p \nonumber \\
&= e^{C_2 \delta} \mathbb{E}
   \left( \lim_{t \to \delta^-}
   \|\partial_x v_N(t, \cdot)\|_{L^2} \right)^p
\le e^{C_2 \delta} \mathbb{E} \|\partial_x v_0\|_{L^2}^p,
\end{align*}
where we used the dominated convergence theorem, \eqref{3.4*} and \eqref{3.6}. Next,
\begin{equation*}
\mathbb{E} \|\partial_x w_N(\delta, \cdot)\|_{L^2}^p
= \mathbb{E} \left( \lim_{t \to 2 \delta^-}
\|\partial_x v_N(t, \cdot)\|_{L^2} \right)^p \le \mathbb{E} \left(
\|\partial_x v_N(\delta, \cdot)\|_{L^2} \right)^p
\le e^{C_2 \delta} \mathbb{E} \|\partial_x v_0\|_{L^2}^p .
\end{equation*}
Similarly,
\begin{equation*}
\mathbb{E} \|\partial_x v_N(2\delta, \cdot)\|_{L^2}^p
 = \mathbb{E} \left( \lim_{t \to 2\delta^-}
\|\partial_x w_N(t, \cdot)\|_{L^2} \right)^p \le e^{C_2\delta}\|\partial_x w_N(\delta, \cdot)\|_{L^2}^p
\le e^{2C\delta} \mathbb{E} \|\partial_x u_0\|_{L^2}^p
\end{equation*}
and
\begin{equation*}
\mathbb{E} \|\partial_x w_N(2\delta, \cdot)\|_{L^2}^p
= \mathbb{E} \left( \lim_{t \to 3\delta^-}
\|\partial_x v_N(t, \cdot)\|_{L^2} \right)^p \le \mathbb{E} 
\|\partial_x v_N(2 \delta, \cdot)\|_{L^2}^p
\le e^{2C_2\delta} \mathbb{E} \|\partial_x u_0\|_{L^2}^p .
\end{equation*}
Therefore, for any $j \in \mathbb{N}$, we get
\begin{equation} \label{3.14}
\|\partial_x u_N(j\delta \cdot)\|_{L^2}^p
\le e^{j C_2 \delta} \mathbb{E} \|\partial_x u_0\|_{L^2}^p,
 {\text{ and }}
\|\partial_x w_N(j\delta, \cdot)\|_{L^2}^p
\le e^{j C_2 \delta} \mathbb{E} \|\partial_x u_0\|_{L^2}^p .
\end{equation}
Next, for any $j = 1,\dots,N$, using Poincar\'e inequality, we obtain
\begin{align}\label{3.15}
\nonumber & \mathbb{E} \sup_{t \in [(j-1)\delta,\, j\delta]}
\|v_N(t)\|_{1,2}^p
\le
\mathbb{E} \| v_N((j-1)\delta+, \cdot) \|_{1,2}^p \\
\nonumber &\le
C(p)\Big(
\mathbb{E}\|\partial_x v_N((j-1)\delta,\cdot)\|_{L^2}^p
+ \mathbb{E}\Big| \int_{L} v_N((j-1)\delta+,x)\,dx \Big|^p
\Big) \\
\nonumber &\le
C(p)\Big(
e^{(j-1)\delta C_2}\|\partial_x u_0\|_{L^2}^p
+ \mathbb{E}\Big| \int_{L} w_N((j-1)\delta-,x)\,dx \Big|^p
\Big) \\
\nonumber &\le
C(p)\Big(
e^{(j-1)\delta C_2}\|\partial_x u_0\|_{L^2}^p
+ \mathbb{E}\Big| \int_{L} w_N((j-2)\delta+,x)\,dx \Big|^p
\Big) \\
&\le
C(p)\Big(
e^{(j-1)\delta C_2}\|\partial_x u_0\|_{L^2}^p
+ \mathbb{E}\Big| \int_{L} v_N((j-1)\delta-,x)\,dx \Big|^p
\Big) \\
\nonumber &\le
C(p)\Big(
e^{(j-1)\delta C_2}\|\partial_x u_0\|_{L^2}^p
+ \mathbb{E}\Big| \int_{L} v_N((j-2)\delta+,x)\,dx \Big|^p
\Big) \\
\nonumber &\le
C(p)\Big(
e^{(j-1)\delta C_2}\|\partial_x u_0\|_{L^2}^p
+ e^{(j-1)\delta C_2} \Big( \int_{L} u_0(x)\,dx \Big)^p 
\Big) \le
C_3 \|u_0\|_{H^1}^{p},
\end{align}
where $C_3>0$ does not depend on $j$.
For $w_N(t)$, using \eqref{3.6}, \eqref{3.7}, \eqref{3.12} and  \eqref{3.14}, we have  
\begin{align}\label{3.16}
\nonumber \mathbb{E} \sup_{t \in [(j-1)\delta,\, j\delta]}
\| w_N(t, \cdot) \|_{1,2}^p
&\le
C_1 e^{C_2\delta}
\mathbb{E}\| w_N((j-1)\delta+, \cdot)\|_{1,2}^p \\
\nonumber &\le
C_p C_1 e^{C_2\delta}
\Big(
\mathbb{E}\|\partial_x w_N((j-1)\delta+)\|_{L^2}^p
+ \mathbb{E}\Big| \int_{L} w_N((j-1)\delta+,x)\,dx \Big|^p
\Big) \\
&\le
C_p C_1 e^{C_2\delta}
\Big(
e^{(j-1)\delta C_2}\|\partial_x u_0\|_{L^2}^p
+ \mathbb{E}\Big| \int_{L} v_N(j\delta-,x)\,dx \Big|^p
\Big) \\
\nonumber &\le
C_p C_1 e^{C_2\delta}
\Big(
e^{(j-1)\delta C_2}\|\partial_x u_0\|_{L^2}^p
+ \mathbb{E}\Big| \int_{L} v_N((j-1)\delta+,x)\,dx \Big|^p
\Big) \\
\nonumber &\le
C_p C_1 e^{C_2\delta}
\Big(
e^{(j-1)\delta C_2}\|\partial_x u_0\|_{L^2}^p
+ \mathbb{E}\Big| \int_{L} w_N((j-1)\delta-,x)\,dx \Big|^p
\Big) \\
\nonumber &\le
C_p C_1 e^{C_2\delta}
\Big(
e^{(j-1)\delta C_2}\|\partial_x u_0\|_{L^2}^p
+ \mathbb{E}\Big| \int_{L} w_N((j-2)\delta+,x)\,dx \Big|^p
\Big) \\
\nonumber &\le
C_p C_1 e^{C_2\delta}\Big(
e^{(j-1)\delta C_2}\|\partial_x u_0\|_{L^2}^p
+ e^{(j-1)\delta C_2} \Big( \int_{L} u_0(x)\,dx \Big)^p 
\Big) \le
C_4 \|u_0\|_{H^1}^{p}.
\end{align}
It follows from \eqref{3.15} and \eqref{3.16} that
\begin{equation*}
\label{3.16*}
\mathbb{E} \sup_{t \in [0,T]} \|u_N(t,\cdot)\|_{1,2}^p
 + \mathbb{E} \sup_{t \in [0,T]} \|w_N(t)\|_{1,2}^p
 + \mathbb{E} \sup_{t \in [0,T]} \|v_N(t)\|_{1,2}^p 
 \le C \, \mathbb{E} \|u_0\|_{1,2}^p.
\end{equation*}
It remains to estimate the integral term in \eqref{3.13}.
To simplify the presentation, denote
\[
Y(t) := 2 \|v_N(t)\|_{1,2}^{p-2}
\int_{\{v_N(t,\cdot)>0\}} v_N^n (\partial_x^3 v_N)^2\, dx .
\]
Using \eqref{3.4*}, on $[(j-1)\delta, j\delta)$, for $j = 1$ we have
\[
\int_{0}^{\delta} Y(t)\, dt
\le \|\partial_x u_0\|_{2}^{p}
- \|\partial_x v_N(\delta-,\cdot)\|_{2}^{p}.
\]
For $j = 2$, we get
\begin{align*}
\mathbb{E} \int_{\delta}^{2\delta} Y(t)\, dt
&\le \mathbb{E}\|\partial_x v_N(\delta+,\cdot)\|_{2}^{p}
- \mathbb{E}\|\partial_x v_N(2\delta-,\cdot)\|_{2}^{p} \\
&= \mathbb{E}\|\partial_x w_N(\delta-,\cdot)\|_{2}^{p}
- \mathbb{E}\|\partial_x v_N(2\delta-,\cdot)\|_{2}^{p} \\
&\le e^{C_2 \delta}
\mathbb{E}\|\partial_x v_N(\delta-,\cdot)\|_{2}^{p}
- \mathbb{E}\|\partial_x v_N(2\delta-,\cdot)\|_{2}^{p}.
\end{align*}
Therefore, we get
\[
\mathbb{E} \int_{\delta}^{2\delta} Y(t)\, dt
\le \|\partial_x u_0\|_{2}^{p}
+ \left(e^{C_2\delta}-1\right)
\mathbb{E}\|\partial_x v_N(\delta-,\cdot)\|_{2}^{p}
- \mathbb{E}\|\partial_x v_N(2\delta-,\cdot)\|_{2}^{p}.
\]
In a similar way, for $j = 3$ we have
\begin{align*}
\mathbb{E} \int_{2\delta}^{3\delta} Y(t)\, dt
&\le \mathbb{E}\|\partial_x v_N(2\delta+,\cdot)\|_{2}^{p}
- \mathbb{E}\|\partial_x v_N(3\delta-,\cdot)\|_{2}^{p} \\
&= \mathbb{E}\|\partial_x w_N(2\delta-,\cdot)\|_{2}^{p}
- \mathbb{E}\|\partial_x v_N(3\delta-,\cdot)\|_{2}^{p} \\
&\le e^{C_2\delta}
\mathbb{E}\|\partial_x v_N(2\delta-,\cdot)\|_{2}^{p}
- \mathbb{E}\|\partial_x v_N(3\delta-,\cdot)\|_{2}^{p}.
\end{align*}
Thus,
\begin{align*}
\mathbb{E} \int_{0}^{3\delta} Y(t)\, dt
&\le \|\partial_x u_0\|_{2}^{p}
+ \left(e^{C_2\delta}-1\right)
\mathbb{E}\|\partial_x v_N(\delta-,\cdot)\|_{2}^{p} \\
&\quad + \left(e^{C_2\delta}-1\right)\mathbb{E}\|\partial_x v_N(2\delta-,\cdot)\|_{2}^{p} - \mathbb{E}\|\partial_x v_N(3\delta-,\cdot)\|_{2}^{p} \\
& \le \|\partial_x u_0\|_{2}^{p} + e^{C_2 T} \left(e^{C_2\delta}-1\right) \|\partial_x u_0\|_{2}^{p} + e^{C_2 T} \left(e^{C_2\delta}-1\right) \|\partial_x u_0\|_{2}^{p} - \mathbb{E}\|\partial_x v_N(3\delta-,\cdot)\|_{2}^{p}.
\end{align*}
Continuing the process on the interval $[(j-1)\delta,j\delta)$,
we obtain \eqref{3.13}, which completes the proof of Proposition \ref{prop:3.5}. 
\end{proof}

For any Banach space $X$ and $D \subset \mathbb{R}^d$, denote $B^{s,p}_q(D,X)$
to be the Besov space, defined e.g.\ in \cite{BerLof} or \cite{Tri}.
The space $B^{s,p}_q(\T;\R)$ will be denoted with $B(\T)$.
We will now establish the analog of Proposition 4.2 in \cite{Gess1}.

\medskip

\noindent
\begin{proposition}

For any $p \geq 2$, $\sigma>0$, $k \in (2 \sigma, \tfrac{2}{p}) \cap (2 \sigma,\tfrac12]$
and $q \in \bigl(\tfrac{2}{k-2\sigma}, \infty\bigr)$,
there exists $C>0$ such that for all $N \ge 1$ we have
\begin{equation*}\label{3.18}
u_N \in L^p(\Omega,\mathcal{F},\mathbb{P};
B^{\frac{k}{2}-\sigma,q}_q([0,T]; B^{\frac12-2k,q}_q(\T))
),
\end{equation*}
and
\begin{equation}\label{3.19}
\mathbb{E}\,
\|u_N\|^p_{B^{\frac{k}{2}-\sigma,q}_q([0,T];
B^{\frac12-2k,q}_q(\T))}
\le
C\,\mathbb{E}\Bigl(
\|u_0\|_{1,2}^p + \|u_0\|_{1,2}^{(k+1)p}
\Bigr).
\end{equation}
\end{proposition}

The proposition will be established for $v_N$ and $w_N$
separately. Let us start with the estimate for $v_N$.

\medskip

\noindent
\begin{lemma}\label{lem:3.7}
For any $p \ge 2$, $\sigma>0$ and $q \ge p$ there exists
a constant $C>0$ such that for all $N \ge 1$, $1 \le j \le N+1$,
and $k \in (0,\tfrac{2}{p})$ we have
\[
v_N \in
B^{\frac{k}{2}-\sigma,q}_q([(j-1)\delta,j\delta);
B^{\frac12-2k,q}_q(\T)),
\]
with 
% --- Reconstructed LaTeX from handwritten page ---

\[
\mathbb{E}\Bigg(
\sum_{j=1}^{N+1}
\|v_N\|_{B^{\frac{k}{2}-\sigma,q}_q([0,T];\, B^{\frac12-2k,q}_q(\mathbb{T}_L))}^p
\Bigg)^{\frac{p}{q}}
\le C\big(\|u_0\|_{L^2}^p + \|u_0\|_{L^2}^{(k+1)p}\big).
\]
\end{lemma}
\begin{proof}
For any $(j-1)\delta \le t_1 \le t_2 < j\delta$ and
$\varphi \in C^\infty(\mathbb{T}_L)$ we have
\[
\bigl(v_N(t_2)-v_N(t_1),\varphi\bigr)_{L^2}
=
\int_{t_1}^{t_2}
\int_{\{v_N>0\}}
v_N^n \,\partial_x^3 v_N \,\partial_x \varphi \, dx\, dt,
\qquad \text{P-a.s.}
\]
This P-almost surely yields
\[
\|v_N(t_2)-v_N(t_1)\|_{H^{-1}(\mathbb{T}_L)}^2
\le
\left(
\int_{t_1}^{t_2}
\Big(
\int_{\{v_N>0\}}
v_N^n (\partial_x^3 u_N)^2 \, dx
\Big)^{1/2}
dt
\right)^2
\]
\[
\le
\left(\int_{t_1}^{t_2}
\sup_{x\in[0,L]} |v_N(x,t)|^{n/2}
\left(
\int_{\{v_N>0\}}
v_N^n (\partial_x^3 v_N)^2 \, dx
\right)^{1/2} dt \right)^{2}
\]
\[
\le
C (t_2-t_1)
\int_{t_1}^{t_2}
\|v_N\|_{1,2}^2
\int_{\{v_N>0\}}
v_N^n (\partial_x^3 v_N)^2 \, dx \, dt,
\]
where we used the Sobolev embedding and the Cauchy--Schwarz inequality.
Thus, using \eqref{3.13}, we have
\[
\|v_N\|_{L^2(\Omega,\mathcal{F},\mathbb{P};
C^{1/2}([(j-1)\delta,j\delta); H^{-1}(\mathbb{T}_L)))}^2
\le
C \|u_0\|_{L^2}^2,
\]
where $C>0$ depends only on $L$.
The rest of the proof of Lemma \ref{lem:3.7} is analogous to the proof
of Lemma~4.3 in \cite{Gess1}.

\end{proof}

We now proceed with the corresponding result for $w_N$.

\begin{lemma}\label{lem:3.8}
    For any $p\in[2,\infty)$, $\ve>0$, and
$q\in[p,\infty)$ there exists $C>0$ such that for all
$N\in\mathbb{N}$, $j\in\{1,\dots,N+1\}$, and $\gamma \in(0,1)$ we have
\[
w_N \in
L^p\!\left(
\Omega,\mathcal{F},\mathbb{P};
B^{\frac{\gamma}{2}-\ve,q}_q([(j-1)\delta,j\delta);
B^{\frac12-2\gamma,q}_q(\mathbb{T}_L))
\right),
\]
with
\[
\mathbb{E}\Bigg(
\sum_{j=1}^{N+1}
\|w_N\|_{B^{\frac12-\ve,q}_q([(j-1)\delta,j\delta);
B^{\frac12-2\gamma,q}_q(\mathbb{T}_L))}^p
\Bigg)
\le
C \|u_0\|_{1,2}^p .
\]
\end{lemma}

\begin{proof}
The proof of Lemma \ref{lem:3.8} is analogous to the proof of Lemma~4.2 in \cite{Gess1}, and hence the proof of Proposition~3.6 follows the lines of Proposition~4.2 in \cite{Gess1}.
\end{proof}

\subsection*{Convergence of the splitting scheme}

We start with the proposition, which is an analog of Proposition~5.2 in \cite{Gess1}.

\begin{proposition}\label{prop:3.7}
Denote
\[
X_u := BC^0([0,T]\times \T), \quad
X_J := L^2([0,T]\times \T)\ \text{(with weak topology)}, \quad
X_\beta := BC^0([0,T];\mathbb{R}).
\]
Then there exist random variables
\[
\tilde u, \tilde u_N : [0,T] \to X_u, \qquad
\tilde J, \tilde J_N : [0,T] \to X_J, \qquad
\tilde \beta_N^{'}, \tilde \beta_N : [0,T] \to X_\beta,
\]
with
\[
(\tilde u_N, \tilde J_N, \tilde \beta_N^{'}) \sim (u_N, J_N, \beta_N),
\]
where
\[
J_N := \chi_{v_N > 0}\, v_N^n(\partial_x^3 v_N).
\]
Furthermore, there exists a subsequence (still indexed by $N$) such that
\[
\tilde u_N(\omega) \to \tilde u(\omega) \quad \text{in } X_u, \qquad
\tilde J_N(\omega) \rightharpoonup \tilde J(\omega) \quad \text{in } X_J,
\]
and
\[
\tilde \beta_N^{'}(\omega) \to \tilde \beta(\omega) \quad \text{in } X_\beta,
\]
for every $\omega \in [0,1]$, as $N \to \infty$.
\end{proposition}

\begin{proof}
Apart from the convergence of $\tilde J_N$, the proofs of the remaining statements are analogous to Proposition~5.2 in \cite{Gess1}.  
In order to show the tightness of $J_N$, the Sobolev embedding theorem and Markov's inequality imply
\[
\mathbb{P}\Bigl( \|J_N\|_{L^2([0,T]\times \T)} > R \Bigr)
\le \frac{1}{R^2} \mathbb{E}\int_0^T \int_{L}
 v_N^{2n}\, |\partial_x^3 v_N|^2 \, dx\, dt
\]
\[
\le \frac{C}{R^2} \mathbb{E}\int_0^T \|v_N\|_{1,2}^n
\int_0^T \int_{L} v_N^n |\partial_x^3 v_N|^2 \, dx\, dt
\le \frac{C}{R^2} \|u_0\|_{1,2}^{n+2} \to 0
\quad \text{as } R \to \infty,
\]
uniformly in $N \in \mathbb{N}$, as follows from \eqref{3.13}.
Hence, the set
\[
\{ \|J_N\|_{L^2([0,T] \times \T)} \leq R \}
\]
is weakly compact in $L^2([0,T]\times \T)$.
We may now conclude the convergence using the Skorokhod--Jakubowski theorem \cite{Jakub}.
\end{proof}
\medskip

In a similar way to Corollary~5.4 in \cite{Gess1}, we have

\begin{corollary}\label{cor:3.8}
For $\tilde u_N$, $\tilde J_N$, $\tilde W_N$ and $\tilde u$ as in Proposition \ref{prop:3.7}, we have
\begin{equation}\label{3.19}
\|\tilde u_N - \tilde u\|_{BC^0([0,T]\times \T)} \to 0,
\quad
\|\tilde v_N - \tilde u\|_{L^\infty([0,T]\times \T)} \to 0,
\quad
\|\tilde w_N - \tilde u\|_{L^\infty([0,T]\times \T)} \to 0,
\tag{3.19}
\end{equation}
as $N \to \infty$, $\tilde{\mathbb{P}}$-almost surely.
\end{corollary}

\medskip

Following  \cite{Gess1}, Proposition 5.5, once again, we get

\begin{corollary}\label{cor:3.9}
Let $\tilde u_N$ and $\tilde u$ be as in Proposition \ref{prop:3.7}.
Then there exist subsequences of $\tilde u_N$, $\tilde v_N$ and $\tilde w_N$, again denoted by
$\tilde u_N$, $\tilde v_N$ and $\tilde w_N$, such that for any $p \geq 2$ we have
\begin{equation*} \label{3.20}
\tilde u_N \rightharpoonup^\ast \tilde u,
\qquad
\tilde v_N \rightharpoonup^\ast \tilde u,
\qquad
\tilde w_N \rightharpoonup^\ast \tilde u,
\quad \text{as } N \to \infty
\end{equation*}
in
\[
L^p\bigl([0,T]; L^\infty([0,T]; H^1(\T))\bigr).
\]
Furthermore,
\begin{equation*}
\tilde{\mathbb{E}} \, \sup_{t \in [0,T]} \|\tilde u(t)\|_{H^1}^p
\le C \|\tilde u_0\|_{H^1}^p,
\end{equation*}
with a constant $C>0$ independent of $\tilde u$ and $\tilde u_0$.
Hence, $\tilde u$ is a bounded continuous $H^1_w(\T)$-valued process, where
$H^1_w(\T)$ is the space $H^1(\T)$ with weak topology induced by $\|\cdot\|_{H^1}$.
\end{corollary}

\medskip

We will also need the analog of Proposition~5.6 from \cite{Gess1}.

\begin{proposition}\label{prop:3.10}
Let $\tilde u_N$, $\tilde u$ and $\tilde J_N$ be as in Proposition \ref{prop:3.7}.
Then the distributional derivative $\partial_x^3 \tilde u$ satisfies
\[
\partial_x^3 \tilde u \in L^2(\{ \tilde u > r \})
\quad \text{for any } r > 0.
\]
Furthermore,
\[
\tilde J_N = \chi_{\tilde v_N > 0}\, \tilde u_N \, (\partial_x^3 \tilde u_N),
\qquad
\text{and}
\qquad
\tilde J = \chi_{\tilde u > 0}\, \partial_x^3 \tilde u,
\quad \tilde{\mathbb{P}}\text{-almost surely}.
\]
\end{proposition}

\medskip

\begin{proof}
The proof follows analogously to Proposition~5.6 in \cite{Gess1}. 
It is based on inequality \eqref{3.13} from Proposition \ref{prop:3.5} and the corresponding estimates for $u_N$.
\end{proof}

\begin{proof}[Proof of Theorem \ref{Th:2.1}.] Analogously to \cite{Gess1}, from the Deterministic--Stochastic connection, for any $t \in [0,T]$ we deduce
\begin{align*}
(v_N(t),\varphi)_2 - (u_0,\varphi)_2
&= (v_N(t),\varphi)_2
+ \sum_{j=1}^{[t/\delta]}
\Big(
- (v_N(j\delta, \cdot),\varphi)_2
+ \lim_{s\to j\delta} (w_N(s, \cdot),\varphi)_2
\Big)
\nonumber\\
&\quad +
\sum_{j=1}^{[t/\delta]}
\Big(
\lim_{s\to j\delta} (v_N(s, \cdot),\varphi)_2
- (w_N((j-1)\delta,\cdot),\varphi)_2
- (v_N(0, \cdot),\varphi)_2
\Big)
\nonumber\\
&= (v_N(t, \cdot),\varphi)_2 - (v_N([t/\delta] \delta, \cdot),\varphi)_2
\nonumber\\
&\quad +
\sum_{j=1}^{[t/\delta]}
\Big(
\lim_{s\to j\delta} (v_N(s, \cdot),\varphi)_2
- (v_N((j-1)\delta, \cdot),\varphi)_2
\Big)
\nonumber\\
&\quad +
\sum_{j=1}^{[t/\delta]}
\Big(
\lim_{s\to j\delta} (w_N(s, \cdot),\varphi)_2
- (w_N((j-1)\delta, \cdot),\varphi)_2
\Big)
\nonumber\\
&= \int_0^t \int_{\{ v_N>0\}}
u_N(s)\, \partial_x^3 v_N(s)\, \partial_x \varphi \, dx\, ds
\nonumber\\
&\quad + \int_0^{[t/\delta] \delta}
\gamma(s)\,(w_N(s, \cdot),\varphi)_2\, ds
+ \int_0^{[t/\delta] \delta}
\alpha(s)\,(w_N(s, \cdot),\varphi)_2\, d\beta(s),
\end{align*}
for all $\varphi \in C^\infty(\mathbb T_L)$, $\mathbb P$-a.s. Changing the stochastic basis to
\[
([0,1], \tilde{\mathcal F}, (\tilde{\mathcal F}_t), \tilde{\mathbb P}),
\]
we obtain the existence of an in-law equivalent convergent subsequences,
denoted with $\tilde u_N$, $\tilde v_N$ and $\tilde w_N$, such that
\begin{align}
(\tilde v_N(t),\varphi)_2 - (u_0,\varphi)_2
&= \int_0^t \int_{\{\tilde v_N(s)>0\}}
\tilde v_N^n(s)\, \partial_x^3 \tilde v_N(s)\, \partial_x \varphi \, dx\, ds
\nonumber\\
&\quad + \int_0^{[t/\delta] \delta}
\gamma(s)\,(\tilde w_N(s, \cdot),\varphi)_2\, ds
+ \int_0^{[t/\delta] \delta}
\alpha(s)\,(\tilde w_N(s, \cdot),\varphi)_2\, d\tilde\beta(s).
\label{3.22}
\end{align}
Passing to the limit as $N \to \infty$, we get
\begin{align*}
(\tilde u(t),\varphi)_2 - (u_0,\varphi)_2
&= \int_0^t \int_{\{\tilde u(s)>0\}}
\tilde u^n(s)\, \partial_x^3 \tilde u(s)\, \partial_x \varphi \, dx\, ds
\nonumber\\
&\quad + \int_0^t \gamma(s)\,(\tilde u(s),\varphi)_2\, ds
+ \int_0^t \alpha(s)\,(\tilde u(s),\varphi)_2\, d\tilde\beta(s).
\end{align*}
The proof of the convergence in \eqref{3.22} is done analogously to
Lemma~5.7 in \cite{Gess1}, using Proposition \ref{prop:3.7} and Proposition \ref{prop:3.10}. Non-negativity of the corresponding solution follows from the non-negativity of $\tilde{v}_N$ (Theorem \ref{Th:3.2}), non-negativity of $\tilde{w}_N$ (Lemma \ref{lem:3.4}), and \eqref{3.19}.
\end{proof}

\section{Global existence of solution on $[0,\infty)$ (Step 2)}\label{Sec 4}
In this section we prove Theorem 2.2. Let us start with an entropy estimate for $v_n$. Denote
\[
G(r) := \frac{1}{(2-n)(1-n)} r^{2-n}.
\]
\begin{lemma}\label{lem:4.1}
For any $t \in [0,T]$ we have
\begin{equation}
\mathbb{E} \int_L G(v_n(t,x))\, dx 
\le C \int_\Omega G(u_0(x))\, dx,
\label{4.1}
\end{equation}
for some constant $C>0$ independent of $N$ and $u_0$.
\end{lemma}

\begin{proof}
For any $j=1,\dots,N+1$ and $t \in [(j-1)\delta, j\delta)$, using \cite{Bern}, Theorem 4,
\begin{equation}
\int_L G(v_n(t,x))\, dx 
\le \int_L G(v_n((j-1)\delta,x))\, dx.
\label{4.2}
\end{equation}
Applying It\^o's formula for $G(w_N)$, we obtain
\begin{align*}
& \int_L G(w_N(t,x))\, dx
= \int_L G(w_n((j-1)\delta,x))\, dx  \\
&\quad + \frac{1}{1-n}  \int_{(j-1)\delta}^{t} \gamma(s) \int_L 
 w_N^{2-n}(s,x)\, dx\, ds + \frac{1}{1-n}  \int_{(j-1)\delta}^{t} \alpha(s) \int_L 
 w_N^{2-n}(s,x)\, dx\, d \beta(s)  \\
&\quad + \frac{1}{2} \int_{(j-1)\delta}^{t} 
\alpha^2(s) \int_L w_N^{2-n}(s,x)\, dx\, ds.
\end{align*}
Thus,
\begin{align*}
\mathbb{E} \int_\Omega G(w_N(t,x))\, dx
&\le \mathbb{E} \int_L G(w_N((j-1)\delta,x))\, dx  \\
&\quad + (n-2)(n-1) \int_{(j-1)\delta}^{t} 
\bigl(\gamma(s) +  \alpha^2(s)\bigr)
\mathbb{E} \int_L G(w_N(s,x))\, dx\, ds.
\end{align*}
By Gronwall's lemma, 
\begin{align}
\mathbb{E}\int_{L} G\bigl(w_N(t,x)\bigr)\,dx
\le
\mathbb{E}\int_{L} G\bigl(w_N((j-1)\delta+,x)\bigr)\,dx \,
e^{C_1\delta},
\label{4.3}
\end{align}
where $C_1>0$ is independent of $N$; here we used the boundedness of the
functions $\gamma(t)$ and $\alpha(t)$ on $[0,T]$. Using estimates \eqref{4.2} and \eqref{4.3}, on $[0,\delta]$ we get
\begin{align}
\mathbb{E}\int_{L} G\bigl(v_N(t,x)\bigr)\,dx
&=
\int_{L} G(u_0(x))\,dx, \qquad \text{ and }
\label{4.4}
\\
\mathbb{E}\int_{L} G\bigl(w_N(t,x)\bigr)\,dx
&\le
\mathbb{E}\int_{L} G\bigl(w_N(0+,x)\bigr)\,dx \, e^{C_1\delta}
\nonumber\\
&=
\mathbb{E}\int_{L} G\bigl(v_N(\delta-,x)\bigr)\,dx \, e^{C_1\delta}
\le
e^{C_1\delta} \int_{L} G(u_0(x))\,dx.
\label{4.5}
\end{align}
On $[\delta,2\delta]$ analogously we have
\begin{align}
& \mathbb{E}\int_{L} G\bigl(v_N(t,x)\bigr)\,dx
\le
\mathbb{E}\int_{L} G\bigl(v_N(\delta+,x)\bigr)\,dx
=
\mathbb{E}\int_{L} G\bigl(w_N(\delta-,x)\bigr)\,dx
\nonumber\\
&\le
e^{C_1\delta}\,
\mathbb{E}\int_{L} G\bigl(w_N(0,x)\bigr)\,dx \le
e^{C_1\delta}\,
\mathbb{E}\int_{L} G\bigl(v_N(\delta-,x)\bigr)\,dx
\le
e^{C_1\delta}\int_{L} G(u_0(x))\,dx,
\label{4.6}
\end{align}
where we used \eqref{4.4}, \eqref{4.5} and the relation
\eqref{3.12}. Next, on $[2\delta,3\delta]$ we get
\begin{align*}
& \mathbb{E}\int_{L} G\bigl(v_N(t,x)\bigr)\,dx
\le
\mathbb{E}\int_{L} G\bigl(v_N(2\delta+,x)\bigr)\,dx
=
\mathbb{E}\int_{L} G\bigl(w_N(2\delta-,x)\bigr)\,dx
\nonumber\\
&\le
e^{C_1\delta}
\mathbb{E}\int_{L} G\bigl(w_N(\delta+,x)\bigr)\,dx = e^{C_1\delta}
\mathbb{E}\int_{L} G\bigl(w_N(2\delta-,x)\bigr)\,dx
\leq 
e^{2C_1\delta}
\int_{L} G(u_0(x))\,dx.
\end{align*}
using \eqref{4.6}. Proceeding by induction, we obtain the desired estimate \eqref{4.1}.
\end{proof}

\medskip
\noindent
\textbf{Proof of Theorem \ref{th:2.2}}
Recall that $v_N(t,x)$ is nonnegative $\mathbb P$-a.s., and by
\eqref{4.1}, for any $t\in[0,T]$ we have
\begin{align*}
\mathbb{E}\int_{L} v_N^{\,2-n}(t,x)\,dx
\le
C \int_{L} u_0^{\,2-n}(x)\,dx
< \infty.
\end{align*}
Thus, using Fatou's lemma and Corollary \ref{3.8} we have
\begin{equation}
\mathbb{E}\int_L \tilde{u}^{2-n}(t,x)\,dx
\le
C \int_L u_0^{\,2-n}(x)\,dx .
\label{4.2}
\end{equation}
Hence, for almost every $t \in [0,T]$, the set
\[
\{ x \in L : \tu(t,x,\omega) = 0 \}
\]
has zero measure, $\mathbb{P}$-a.s. By Theorem~3.1 we have
\[
v_N \in C^{\frac18, \frac12}\bigl([(j-1)\delta,j\delta) \times \T\bigr), \text{ and }
W_N \in C\bigl([(j-1)\delta,j\delta) \times \T\bigr),
\]
and hence $u \in C(Q_T).$ In this case \eqref{3.19} implies that $\tu$ is continuous in $t$, and the Sobolev embedding $H^1(\T) \subset C^{1/2}(\mathbb{T}_L$ yield the Hölder continuity of $\tu$  in  $x$. In view of Theorem \ref{2.1}, there exists a filtered probability space
\[
(\Omega^1:=[0,1], \mathcal{F}^1, (\mathcal{F}^1_t)_{t\in[0,T]}, \mathbb{P}^1),
\]
on which there exists a martingale solution $u^1(t)$ of \eqref{1.1}
with a Wiener process $\beta_1(t)$ and filtration
\[
\mathcal{F}^1_t = \sigma\{\beta_1(s),\, s\le t,\ t\in[0,T]\}.
\]
Such solution is nonnegative and satisfies
\[
\mathbb{E}\, \sup_{t\in[0,T]}
\|u^1(t,\cdot)\|_{1,2}^p
\le
C \|u_0\|_{1,2}^p.
\]
This, in turn, implies that $u^1(T, \cdot, \omega) \in H^1(\T)$ almost surely, as well as $u^1(T, \cdot, \omega) \neq 0$ a.s. since 
\[
\{x \in L : u^1(T,x,\omega)=0\}
\text{ has zero measure, } \mathbb{P}^1\text{-a.s.}
\]
We next define a new probability space
$(\Omega', \mathcal{F}', \mathbb{P}')$
and introduce a new Wiener process $\beta'(t,\omega')$ on it.
Finally, define the product probability space as
\[
(\Omega^2, \mathcal{F}^2, \mathbb{P}^2)
:=
(\Omega^1 \times \Omega',
 \mathcal{F}^1 \otimes \mathcal{F}',
 \mathbb{P}^1 \otimes \mathbb{P}').
\]
The Wiener process on this product space is
\[
\beta_2(t,\omega_1,\omega')
:=
\begin{cases}
\beta_1(t,\omega_1), & t \in [0,T], \\[6pt]
\beta'(t-T,\omega') + \beta_1(T,\omega_1),
& t \in [T,2T].
\end{cases}
\]
For every fixed $\omega_1$, the equation \eqref{1.1}  
has a martingale solution $\eta(t,\omega^1)$ on $[T,2T)$ with the initial condition
\[
\eta(T,\omega') = u^{1}(T,\omega_1).
\]
In the same way as before we conclude that $\eta(t,\omega')$ is nonnegative,
\[
\eta(2T, x, \omega') \neq 0 \quad \text{ and } \quad \eta(2T,\cdot, \omega') \in H^{1}(\T) \quad \text{a.s.}
\]
We may now define the solution
\[
u^{2}(t,\omega_1,\omega') :=
\begin{cases}
u^{1}(t,\omega_1), & t \in [0,T), \\[4pt]
\eta(t,\omega'), & t \in [T,2T]
\end{cases}
\]
on the space $(\Omega^{2}, \mathcal{F}^{2}, \mathcal{F}^{2}_{t}, \mathbb{P}^{2})$, 
where the filtration $\{\mathcal{F}^{2}_{t}\}$ is defined by
\[
\mathcal{F}^{2}_{t} = \sigma\big( \beta_{2}(s), u^{2}(s), s \le t \big),
\qquad t \in [0,2T].
\]

Repeating this process $n$ times, we obtain a continuous in time solution
on $[0,nT]$ for any $n \ge 1$. The finite-dimensional distributions of the solutions for each $n \ge 1$ match by construction. It remains to make use of Kolmogorov's theorem, which guarantees
the existence of a probability space and a weak solution $u(t)$, $t \ge 0$,
defined on it. This completes the proof of Theorem \ref{th:2.2}.

\bigskip

\section{Positivity of the solution (Step 3)}\label{Sec 5}

\begin{proposition}\label{prop:5.1}
Suppose $n \ge 4$ and $u_0 \in H^{1}(\mathbb{T}_L)$ with
$u_0(x) > 0$ on $[0,1]$. Then there exists a set
$A \subset \Omega$ with $\mathbb{P}(A)=1$ such that
for all $\omega \in A$,
\[
u(t,x,\omega) > 0
\quad \text{for all } t \ge 0, \; x \in L.
\]
\end{proposition}

\begin{proof}
In order to establish this proposition, we need to improve the Lemma \ref{lem:4.1} estimate on $u^{n}(t,x,\omega)$. Fix a an arbitrary interval $[0,T]$. According to Theorem \ref{th:2.2}, the equation \eqref{1.1} admits a weak martingale solution on $[0,\infty)$. Moreover, this solution is non-negative, and for every $t>0$ the set
\[
\{ x \in L : u(t,x,\omega)=0 \}
\]
has zero Lebesgue measure $\mathbb{P}$ - almost surely.
For $n \ge 4$, consider the approximating sequences
$v_N$ and $w_n$, constructed in Theorem \ref{Th:2.1}.

For any $j=1,\dots,N+1$ and
$t \in [(j-1)\delta,j\delta)$, by \eqref{4.2} we have 
\begin{equation*}
\int_{L} v_N^{2-n}(t,x)\,dx
\le
\int_{L} v_N^{2-n}((j-1)\delta,x)\,dx.
\label{5.1}
\end{equation*}
Furthermore, applying It\^o's formula,
\begin{align*}
\int_{L} w_N^{2-n}(t,x)\,dx
&=
\int_{L} w_N^{2-n}((j-1)\delta+,x)\,dx \notag \\
&\quad
+ \int_{(j-1)\delta}^{t}
\Big[(2-n)\gamma(s)
+ \frac{(1-n)(2-n)}{2}\alpha^2(s)\Big]
\int_{L} w_N^{2-n}(s,x)\,dx\,ds \notag \\
&\quad
+ \int_{(j-1)\delta}^{t}
(2-n)\alpha(s)
\int_{L} w_N^{2-n}(s,x)\,dx\, d\beta(s).
\label{5.2}
\end{align*}
Therefore, $\displaystyle \int_{L} w_n^{2-n}(t,x)\,dx$
satisfies a linear stochastic differential equation.
Hence,
\begin{align*}
&\int_{L} w_N^{2-n}(t,x)\,dx \\
& =
\int_{L} w_N^{2-n}((j-1)\delta+,x)\,dx  \notag 
\quad \times
\exp\Bigg\{
\int_{(j-1)\delta}^{t}
\varphi(s)\,ds
+
\int_{(j-1)\delta}^{t}
\psi(s)\,d\beta(s)
\Bigg\},
\end{align*}
where
\[
\f(s): = (2-n)\gamma(s)
+ \frac{(1-n)(2-n)}{2}\alpha^2(s) - \frac{(2-n)\alpha(s)}{2},
\]
and 
\[
\psi(s):= (2-n) \alpha(s).
\]
Consequently, using the argument of Lemma \ref{lem:4.1}, on $[0, \delta)$ we obtain
\begin{equation*}
\int_{L} u_N^{2-n}(t,x)\,dx
\le
\int_{L} u_0^{2-n}(x)\,dx.
\end{equation*}
Hence,
\begin{align*}
    \int_{L} w_N^{2-n}(t,x)\,dx
& = \int_{L} w_N^{2-n}(0,x)\,dx
\exp\left( \int_0^t \varphi(s)\,ds + \int_0^t \psi(s)\,d\beta(s) \right) \\
& = \int_{L} v_N^{2-n}(\delta-,x)\,dx
\exp\left( \int_0^t \varphi(s)\,ds + \int_0^t \psi(s)\,d\beta(s) \right) \\
& \leq \int_{L} u_0^{2-n}(x)\,dx
\exp\left( \int_0^t \varphi(s)\,ds + \int_0^t \psi(s)\,d\beta(s) \right).
\end{align*}
In a similar way, on $[\delta,2\delta)$ we obtain
\begin{align*}
& \int_{L} v_N^{2-n}(t,x)\,dx \le \int_{L} v_N^{2-n}(\delta+,x)\,dx = \int_{L} w_N^{2-n}(\delta-,x)\,dx  \\
& = \int_{L} w_N^{2-n}(0,x)\,dx
\exp\left( \int_0^\delta \varphi(s)\,ds + \int_0^\delta \psi(s)\,d\beta(s) \right)\\ 
& =  \int_{L} v_N^{2-n}(\delta-,x)\,dx
\exp\left( \int_0^\delta \varphi(s)\,ds + \int_0^\delta \psi(s)\,d\beta(s) \right)\\
& \le \int_{L} u_0^{2-n}(x)\,dx
\exp\left( \int_0^\delta \varphi(s)\,ds + \int_0^\delta \psi(s)\,d\beta(s) \right) \\
& \le \int_{L} u_0^{2-n}(x)\,dx
\exp\left( \int_0^t \varphi(s)\,ds + \int_0^t \psi(s)\,d\beta(s) \right).
\end{align*}
On $[2\delta,3\delta]$ we get
\begin{align*}
& \int_{L} v_N^{2-n}(t,x)\,dx \le \int_{L} v_N^{2-n}(2\delta+,x)\,dx = \int_{L} w_N^{2-n}(2\delta-,x)\,dx  \\
& = \int_{L} w_N^{2-n}(\delta+,x)\,dx
\exp\left( \int_\delta^t \varphi(s)\,ds + \int_\delta^t \psi(s)\,d\beta(s) \right)\\ 
& =  \int_{L} u_N^{2-n}(2\delta-,x)\,dx
\exp\left( \int_\delta^t \varphi(s)\,ds + \int_\delta^t \psi(s)\,d\beta(s) \right)\\
& \le \int_{L} u_0^{2-n}(x)\,dx
\exp\left( \int_0^t \varphi(s)\,ds + \int_0^t \psi(s)\,d\beta(s) \right).
\end{align*}
Therefore, for any $t \in [(j-1)\delta, j\delta)$,
\[
\int_{L} v_N^{2-n}(t,x)\,dx
\le
\int_{L} u_0^{2-n}(x)\,dx
\exp\left( \int_0^t \varphi(s)\,ds + \int_0^t \psi(s)\,d\beta(s) \right).
\]
Hence,
\[
\sup_{t \in [0,T]}
\int_{L} v_N^{2-n}(t,x)\,dx
\le
\int_{L} u_0^{2-n}(x)\,dx
\sup_{t \in [0,T]}\left[
\exp\left( \int_0^t \varphi(s)\,ds + \int_0^t \psi(s)\,d\beta(s) \right) \right].
\]
Next, using the same argument as in the proof of Lemma \ref{lem:3.4}, the random process
\[
\xi(t)
:=
\exp\left(
\int_0^t \varphi(s)\,ds
+
\int_0^t \psi(s)\,d\beta(s)
\right)
\]
is a solution of a linear equation
\[
\xi(t)
= 1
+ \int_{0}^{t}
\Big[
(2-n) \gamma(s)
+ \frac{(2-n)(1-n)}{2}\,\alpha^2(s)
\Big] \xi(s)\,ds
+ \int_{0}^{t} (2-n)\, \alpha(s) \,\xi(s)\,d\beta(s).
\]
This way,
\[
\mathbb{E}\,\sup_{t\in[0,T]} \xi(t)
\le
\Big(\mathbb{E}\,\sup_{t\in[0,T]} \xi^2(t)\Big)^{1/2}
\le C(T),
\]
where we used the martingale properties of stochastic integrals.
Consequently,
\[
\mathbb{E}\,\sup_{t\in[0,T]}
\int_{L} v_N^{2-n}(t,x)\,dx
\le
\int_{L} u_0^{2-n}(x)\,dx \; C(T).
\]
Using Fatou's lemma, Corollary \ref{cor:3.8}, and Theorem \ref{th:2.2}, we obtain
\begin{equation}\label{5.5}
\mathbb{E}\,\sup_{t\in[0,T]}
\int_{L} u_N^{2-n}(t,x)\,dx
\le
\int_{L} u_0^{2-n}(x)\,dx \; C(T).
\end{equation}
In order to complete the proof of Proposition \ref{prop:5.1}, we now argue by contradiction. Assume there exists a set $B \subset \Omega$ such that $P(B)>0$ and
for any $\omega\in B$ there exist $t(\omega)$, $x_0(\omega)$ such that
\[
u(t(\omega),x_0(\omega),\omega)=0.
\]
Using the embedding $H^1(\T) \subset C^{1/2}(\T)$, % (?)
we have
\[
|u(t(\omega),x,\omega) - u(t(\omega),x_0(\omega),\omega)|
\le
K(t,\omega)\,|x-x_0(\omega)|^{1/2}.
\]
Consequently, 
\begin{align}\label{5.6}
 \nonumber &  \mathbb{E}\,\sup_{t\in[0,T]}
\int_{K} u^{2-n}(t,x)\,dx
\ge
\int_{\{\omega\in B\}}
\sup_{t\in[0,T]}
\int_{L} u^{2-n}(t,x,\omega)\,dx \; P(d\omega) \\
& \ge \int_{\{\omega\in B\}}
\int_{L} u^{2-n}(t(\omega),x,\omega)\,dx \; P(d\omega) \\
\nonumber & \ge
\int_{\{\omega\in B\}}
\frac{1}{K^{\,2-n}(t,\omega)}
\int_{L}
\frac{dx}{|x-x_0(\omega)|^{\frac{n-2}{2}}}
\,P(d\omega)
= \infty.
\end{align}
The latter integral is divergent since $n \ge 4$. Thus, \eqref{5.6} contradicts   \eqref{5.5}, and the proof of Proposition \ref{prop:5.1} is complete.
\end{proof}
\medskip

\begin{corollary} \label{cor:5.2}
It follows from Proposition \ref{prop:5.1} that if $n\ge 4$, the term
\[
\int_0^t \int_{\{\tilde{u}(s)>0\}}
\tilde{u}^n(s) \partial_x^3 \tilde{u}(s)\,\partial_x \varphi \, dx\,ds,
\]
which appears in Definition \ref{Def:2.1}, formula \eqref{2.1}, may now be replaced with the integral over the entire interval, i.e.
\[
\int_0^t \int_{L}
\tilde{u}^n(s) \partial_x^3 \tilde{u}(s)\,\partial_x \varphi \, dx\,ds.
\]
\end{corollary}

\section{Proof of the main results (Step 4)} \label{Sec 6}

Let us start with the proof of Remark \ref{rem:1}, dealing with the long time behavior of the equation \eqref{1.2*}.
\begin{proof}[Proof of Remark \ref{rem:1}] Introduce 
\begin{equation}\label{1*}
v(t,x):=u(t, x - \beta(t)),
\end{equation}
and consider the equation, satisfied by $v$. In order to do that, we note that the change of variables \eqref{1*} is equivalent to 
\[
u(t,x) = v(t, x + \beta(t)).
\]
Thus, by Ito's formula,
\[
du(t,x) = [v_t(t,x+\beta(t)) + \frac{1}{2}v_{xx}(t,x+\beta(t))] dt + \d_x v(t,x+\beta(t)) d\beta(t).
\]
The equation \eqref{1.2*} now reads as

\begin{multline*}
v_t(t,x+\beta(t)) dt + \frac{1}{2}v_{xx}(t,x+\beta(t)) dt + \d_x v(t,x+\beta(t)) d\beta(t)  = \\ \left(\partial_x(-v(t,x+\beta(t))^{2}v_{xxx}(t,x+\beta(t))) + \frac{1}{2}  v_{xx}(t,x+\beta(t)) \right)dt +
\partial_x v(t,x+\beta(t)) \, d \beta(t),
\end{multline*}
or
\[
v_t(t,x+\beta(t))  = \partial_x(-v(t,x+\beta(t))^{2}v_{xxx}(t,x+\beta(t)).
\]
After re-labling $x + \beta(t) = y$, the above equation becomes
\begin{equation*}
v_t(t,y)  = \partial_y(-v(t,y)^{2}v_{yyy}(t,y)).
\end{equation*}
Since $\beta(0) = 0$, the initial conditions for $u$ and $v$ are the same, we may conclude, following \cite{DalPas, BertPugh}, that
\begin{equation*}\label{1.4}
 \lim_{t \to \infty}v(t,y) = \frac{1}{L} \int_0^L u_0(x) \, dx,
 \end{equation*}
or, in our original notation, 
\[
\frac{1}{L} \int_0^L u_0(x) \, dx =  \lim_{t \to \infty}v(t,y)  =  \lim_{t \to \infty}v(t,x+\beta(t)) = \lim_{t \to \infty}u(t,x).
\]
Furthermore, since, by periodicity,
\[
\int_{0}^{L} v_y^2(t,y) \, dy = \int_{0}^{L} u_x^2(t,x-\beta(t)) \, dx  = \int_{0}^{L} u_x^2(t,x) \, dx:= J[u(t)], 
\]
the exponential decay of the energy $J[u(t)]$ follows from the main result of \cite{Tud}, applied to $v$.
\end{proof}

We proceed with the proof of Theorem \ref{Th:2.3}. To this end, we will use the result of Khasminskii \cite{Khas} on the stability of linear stochastic equations
\begin{equation}
dX(t) = b(t)X(t)\,dt + \sigma(t)X(t)\,d\beta(t).
\label{6.1}
\end{equation}

\begin{theorem}(\cite{Khas}, Ch.~6) \label{Th:6.1.1}
Consider the equation \eqref{6.1}.

\medskip
\noindent
(i) If
\[
\int_0^\infty \sigma^2(t)\,dt < \infty,
\]
then the condition
\[
\int_0^\infty b(t)\,dt = -\infty
\]
is necessary and sufficient for almost sure asymptotic stability of the trivial solution of \eqref{6.1}.

\medskip
\noindent
(ii) If
\[
\int_0^\infty \sigma^2(t)\,dt = \infty,
\]
then the condition
\begin{equation}
\limsup_{t\to\infty}
\frac{\displaystyle \int_0^t \left(b(s)-\frac12\sigma^2(s)\right)\,ds}
{\displaystyle \left(2\int_0^t \sigma^2(s)\,ds \,
\ln\ln\!\left(\int_0^t \sigma^2(s)\,ds\right)\right)^{1/2}}
< -1
\label{ast*ast}
\end{equation}
is sufficient for almost sure asymptotic stability.
\end{theorem}

\bigskip

\noindent
\textbf{Proof of Theorem \ref{Th:2.3}.} As we showed earlier, the equation for the mass has the form \eqref{6.1}, i.e. 
\begin{equation}\label{mass}
\int_L u(t,x) \, dx = \int_L u_0(x) \, dx + \int_0^t \gamma(s) \int_L u(s,x)\, dx \, ds + \int_0^t  \alpha(s) \int_L u(s,x)\, dx \, d \beta(s),
\end{equation}
whose solution is 
\begin{equation}\label{6.2}
\int_L u(t,x) \, dx = \int_L u_0(x) \, dx \exp\left\{
\int_0^t \left(\gamma(s)-\frac12\alpha^2(s)\right)\,ds
+ \int_0^t \alpha(s)\,d\beta(s)
\right\}.
\end{equation}
Using It\^o's formula for the solution of \eqref{1.1} and Corollary \eqref{cor:5.2}, we obtain
\begin{align}\label{6.3}
\|\d_x u(t, \cdot)\|_{L^2}^2
&= \|\d_x u_0\|_{L^2}^2
- 2 \int_0^t \int_L u^n (\d_x^3 u)^2 \, dx \,ds
+ 2 \int_0^t \gamma(s) \|\d_x u(s, \cdot)\|_{L^2}^2\,ds  \nonumber \\
&\quad + 2 \int_0^t \alpha(s)\|\d_x u(s, \cdot)\|_{L^2}^2\,d\beta(s)
+ \int_0^t \alpha^2(s)\|\d_x u(s, \cdot)\|_{L^2}^2\,ds .
\end{align}
Using the comparison theorem for stochastic differential equations (see, e.g. \cite{VI}), we get
\begin{equation}\label{6.4}
\|\d_x u(t, \cdot)\|_{L^2}^2 \le y(t), \qquad t \ge 0,
\end{equation}
where $y(t)$ is the solution of the  equation 
\begin{equation}\label{6.5}
y(t) = \|\d_x u_0\|_{L^2}^2 
+ \int_0^t (2 \gamma(s) + \alpha^2(s))y(s)\,ds
+ \int_0^t 2\alpha(s)y(s)\,d\beta(s).
\end{equation}
Solving \eqref{6.5}, we obtain
\begin{equation*}\label{6.6}
y(t)
= \|\d_x u_0\|_{L^2}^2
\exp\left\{
\int_0^t \left(2 \gamma(s)-\alpha^2(s)\right)\,ds
+\int_0^t 2 \alpha(s)\,d\beta(s)
\right\}.
\end{equation*}
Taking into account the first statement of Theorem \ref{Th:6.1.1}, under the condition \eqref{2.4} we obtain
\begin{equation}
\int_L u(t,x,\omega)\,dx \to 0,
\qquad
y(t,\omega) \to 0,
\qquad t\to\infty,
\quad \text{a.s.}
\label{6.7}
\end{equation}
In view of  \eqref{6.4}, \eqref{6.7} implies
\begin{equation}
\|\d_x u(t, \cdot, \omega)\|_{L^2}^2 \to 0,
\qquad t\to\infty,
\quad \text{a.s.}
\label{6.8}
\end{equation}
Moreover, 
\begin{align*}
\|u(t,\cdot,\omega)\|_\infty
&\le \|\overline{u(t, \cdot, \omega)}\|_{\infty}
+ \|u(t,\cdot, \omega)-\overline{u(t, \cdot, \omega)}\|_\infty \nonumber\\
&\le C\|\d_x u(t, \cdot, \omega)\|_{L^2}^2
+ \frac{1}{L}
\int_L u(t,x,\omega)\,dx,
\end{align*}
which, in view of \eqref{6.7} and \eqref{6.8}, implies
\[
\|u(t,\cdot,\omega)\|_\infty \to 0,
\qquad t\to\infty,
\quad \text{a.s.}
\]
Next, using \eqref{6.2}, we have
\begin{multline*} \label{6.9}
\mathbb{E}\!\left(\int_L u(t,x)\,dx\right)^2
\\
=
\left(\int_L u_0(x)\,dx\right)^2
\exp\left\{
2\int_0^t \left(\gamma(s)-\frac12\alpha^2(s)\right)\,ds
+ 2\int_0^t \alpha^2(s)\,ds
\right\} \to 0 \\
\text{ as } t \to \infty.
\end{multline*}
Using \eqref{6.4} and \eqref{6.5}, we have
\begin{equation}\label{6.10}
\mathbb{E}\|\d_x u(t, \cdot)\|_{L^2}^2
\le \mathbb{E}y(t)
= \|\d_x u_0\|_{L^2}^2
\exp\left\{
\int_0^t (2 \gamma(s) + \alpha^2(s)) \,ds
\right\} \to 0 \text{ as } t \to \infty.
\end{equation}
Thus,
\begin{align*}
\mathbb{E}\|u(t, \cdot)\|_\infty^2
&\le 2\Big(
\mathbb{E}\|u(t, \cdot)-\overline{u(t)}\|_\infty^2
+ \mathbb{E}\|\overline{u(t)}\|_{\infty}^2
\Big) \nonumber\\
&\le C\Big(
\mathbb{E}\|\d_x u(t, \cdot)\|_{L^2}^2
+ \mathbb{E}\Big(\int_L u(t,x)\,dx\Big)^2
\Big)
\to 0,
\qquad t\to\infty.
\end{align*}
This completes the proof of the first part of Theorem \ref{Th:2.3}.

\bigskip

For the proof of the second part of Theorem \ref{Th:2.3},  we calculate the experession \eqref{ast*ast} for the equation \eqref{mass}, satisfied by the mass, as well as for the equation \eqref{6.5}, satisfied by $y$.  
For \eqref{mass} we have
\begin{equation*}\label{6.11}
\frac{
\displaystyle
\int_0^t \left(\gamma(s)-\frac12\alpha^2(s)\right)\,ds
}{
\displaystyle
\left(
2\int_0^t \alpha^2(s)\,ds
\ln\ln\!\left(\int_0^t \alpha^2(s)\,ds\right)
\right)^{1/2}
}
:= I_1,
\end{equation*}
while for the equation for $y(t)$ we obtain
\begin{equation*}\label{6.12}
\frac{
\displaystyle
\int_0^t \left(\gamma(s)-\frac12\alpha^2(s)\right)\,ds
}{
\displaystyle
\left(
2\int_0^t \alpha^2(s)\,ds
\ln\ln\!\left(\int_0^t 4 \alpha^2(s)\,ds\right)
\right)^{1/2}
}
:= I_2.
\end{equation*}
Comparing $I_1$ and $I_2$, we see that if $I_2$ satisfies \eqref{2.6}, then
\[
\liminf_{t\to\infty} I_1(t) < -1
\]
holds. Hence,  Theorem \ref{Th:2.3} follows from Theorem \ref{Th:6.1.1}.

\bigskip

\noindent
\textbf{Proof of Theorem \ref{Th:2.5}.}

In order to proceed, we need the following probabilistic analog of Lemma~1 in \cite{Tud}.

\begin{lemma}
There is a constant $C>0$ such that for any non-negative measurable functions $w(x,\omega)$ and $v(x,\omega) \in H^3(\T)$ we have
\begin{equation}\label{6.13}
\mathbb{E}
\left(
\int_L \frac{v^2(x,\omega)}{w(x,\omega)}\,dx
\right)
\;
\mathbb{E}
\left(
\int_L w(x,\omega)\,
\big(\partial_x^3 v(x,\omega)\big)^2\,dx
\right)
\ge
C
\left(
\mathbb{E}
\int_L (\partial_x^2 v(x,\omega))^2\,dx
\right)^2.
\end{equation}

\end{lemma}

\begin{proof}
For any random variable $x_0(\omega) \in L$ and $x \in L$, 
using the Schwarz inequality, we obtain
\begin{align*}
\left|
\mathbb{E}
\int_{x_0}^x v\,\partial_x^3 v\,dx
\right|
& \leq \mathbb{E}
\int_L 
\frac{v}{\sqrt{w}}
\sqrt{w}\,|\partial_x^3 v|\,dx \\
&\le
\left(
\mathbb{E}
\int_L \frac{v^2}{w}\,dx
\right)^{1/2}
\left(
\mathbb{E}
\int_L w(\partial_x^3 v)^2\,dx
\right)^{1/2}
=: A.
\end{align*}
In other words,
\begin{equation*}\label{6.14}
A \geq  - \E \int_{x_0}^x v\,\partial_x^3 v\,dx.
\end{equation*}
On the other hand, integration by parts yields
\begin{multline*}
-\int_{x_0}^x v\,\partial_x^3 v\,dx
= - v(x,\omega)\,\partial_x^2 v(x, \omega) + v(x_0,\omega)\,\partial_x^2 v(x_0, \omega) + \frac{1}{2} (\d_x v(x, \omega))^2 - \frac{1}{2} (\d_x v(x_0, \omega))^2.
\end{multline*}
Since $v, v_x, v_{xx}, v_{xxx}\in L^2(\T)$, there must exist $x_0(\omega) \in \T$ such that $\d_x v(x_0, \omega) = 0$ and $\d_x^2 v(x_0,\omega) \geq 0$. For such $x_0$ the latter inequality yields
\[
A \ge
\mathbb{E}\!\left(
- v(x,\omega)\,\partial_x^2 v(x,\omega) \right)
+ \frac12 \mathbb{E}\bigl(\partial_x v(x,\omega)\bigr)^2
.
\]
Integrating the last inequality over $[0,L]$ and using Fubini's theorem, we obtain
\[
A \ge \frac32 \mathbb{E}\int_L (\partial_x v)^2\,dx,
\]
which completes the proof of the Lemma.
\end{proof}

\medskip

We return to the proof of Theorem \ref{Th:2.5}.
Using \eqref{6.10}, for all $t\ge 0$ we obtain
\begin{equation}\label{6.15}
\mathbb{E}\bigl(\|\partial_x u(t, \cdot)\|_2^2\bigr)
\le \|u_0\|_2^2
\exp\!\left(
\int_0^t [2\gamma(s)+\alpha^2(s)]\,ds
\right)
:= C(u_0).
\end{equation}
Now, for the entropy $\displaystyle \int_L u^{2-n}(t,x)\,dx$, using It\^o's formula, we get
\begin{align*}
\int_L u^{2-n}(t,x)\,dx
&= \int_L u_0^{2-n}(x)\,dx
- \int_0^t \!\!\int_L (\d_x u)^2\, dx \, ds \nonumber\\
&\quad + \int_0^t \left[(2-n) \gamma(s) + \frac{(1-n)(2-n)}{2} \alpha^2(s)\right]\,\!\!\int_L u^{2-n}(s,x),dx\,ds  \nonumber\\
&\quad + \int_0^t (2-n)\alpha(s)
\int_L u^{2-n}(s,x)\,dx\, d\beta(s).
\end{align*}
Using the comparison theorem once again, we have
\begin{equation*}\label{6.16}
\int_L u^{2-n}(t,x)\,dx \le z(t), \qquad t\ge0,
\end{equation*}
where $z(t)$ is the solution of the equation
\begin{align*}
z(t)
&= \int_L u_0^{2-n}(x)\,dx
+ \int_0^t
\Bigl[(2-n)\gamma(s)
+ \frac{(1-n)(2-n)}{2}\alpha^2(s)\Bigr]
z(s)\,ds \\
&\quad + \int_0^t (2-n)\alpha(s) z(s)\,d\beta(s).
\end{align*}
Solving this linear equation, we obtain
\[
\mathbb{E} z(t)
= \int_L u_0^{2-n}(x)\,dx\,
\exp\!\left(
\int_0^t
\Bigl[(2-n)\gamma(s)
+ \frac{(1-n)(2-n)}{2}\alpha^2(s)\Bigr] ds
\right)
\le C_1(u_0)
\]
for any $t\ge0$, due to the condition \eqref{2.7}. 
Therefore, 
\begin{equation}\label{6.17}
\mathbb{E}\int_L u^{2-n}(t,x)\,dx \le C_1(u_0).
\end{equation}
From \eqref{6.3} we have
\begin{align}\label{6.18}
\mathbb{E}\|\partial_x u(t, \cdot)\|_2^2
&= \|\partial_x u_0\|_2^2
- 2 \int_0^t \mathbb{E}\int_L u^n (\partial_x^3 u)^2 \, dx\, ds  \nonumber\\
&\quad + \int_0^t [2\gamma(s)+\alpha^2(s)]
\mathbb{E}\|\partial_x u(s)\|_2^2 \, ds.
\end{align}
Taking into account
that with probability 1 we have $u(t,x,\omega)>0$ for all $t\ge0$, and $x\in L$, the Definition \ref{Def:2.1} yields $u(t,\cdot,\omega) \in H^3(L)$. The inequality \eqref{6.13} with $w:=u^n$ and $v:=u$ yields
\begin{equation*}\label{6.19}
\mathbb{E}\int_L u^n(t,x)\,dx \;
\mathbb{E}\int_L u^n(t,x) (\d_x^3 u(t,x))^2\,dx
\ge C_2(u_0)\bigl(\mathbb{E}\|\partial_x u(t,\cdot)\|_2^2\bigr)^2 .
\end{equation*}
In view of \eqref{6.17}, we have
\begin{equation*}
\mathbb{E}\int_L u^n(t,x)(\partial_x^2 u(t,x))^2\,dx
\ge C_2(u_0)\bigl(\mathbb{E}\|\partial_x u(t)\|_2^2\bigr)^2,
\end{equation*}
thus \eqref{6.18} implies
\begin{equation}\label{6.20}
\frac{d}{dt}\mathbb{E}\|\partial_x u(t, \cdot)\|_2^2
\le (2\gamma(t)+\alpha^2(t))\mathbb{E}\|\partial_x u(t,\cdot)\|_2^2
- C_2(u_0)\bigl(\mathbb{E}\|\partial_x u(t,\cdot)\|_2^2\bigr)^2.
\end{equation}
Using \eqref{6.15} we may conclude that 
$\mathbb{E}\|\partial_x u(t, \cdot)\|_2^2$ is bounded for $t\ge0$.
Thus, taking into account \eqref{2.7ast}, we obtain
\begin{equation}\label{6.21}
(2\gamma(t)+\alpha^2(t))J(t) - C_2(u_0)J^2(t) < 0,
\quad \text{for sufficiently large } t.
\end{equation}
Denoting $J(t):=\mathbb{E}\|\partial_x u(t, \cdot)\|_2^2$, we claim that
$J(t)\to 0$ as $t\to\infty$.
Arguing by contradiction, assume $J(t)\not\to 0$. Then, using \eqref{6.20} and \eqref{6.21} $J(t)$ is
monotone decreasing and bounded from below, hence $J(t)\to J^*$ as $t\to\infty$, for some $J^*>0$.
Therefore,
\[
(2\gamma(t)+\alpha^2(t))J(t) - C_2(u_0)J^2(t)
\to - C_2(u_0)(J^*)^2 \text{ as } t \to \infty.
\]
In other words, for sufficiently large $T$ we have 
\[
(2\gamma(t)+\alpha^2(t))J(t) - C_2(u_0)J^2(t)
\le -\frac{C_2(u_0)}{2}(J^*)^2,
\quad t\ge T .
\]
Hence,
\[
\frac{d}{dt}J(t)
\le -\frac{C_2(u_0)}{2}(J^*)^2,
\quad t\ge T .
\]
This implies $J(t)\to -\infty$ as $t\to\infty$, which is impossible.
Therefore, $J(t)\to 0$ as $t\to\infty$. 
Now, using Poincar\'e's inequality and the embedding
\[
H^1(\mathbb{T}_L) \hookrightarrow L^\infty(\mathbb{T}_L),
\]
we have
\begin{equation*}
\mathbb{E}\|u(t,x)-\bar u_0 \eta(t)\|_{L^\infty}^2
\le C\,\mathbb{E}\|u(t,x)-\bar u_0 \eta(t)\|_{H^1}^2 \le C_2\,\mathbb{E}\|\partial_x u(t, \cdot)\|_2^2
\longrightarrow 0,
\qquad t\to\infty,
\end{equation*}
which completes the proof of Theorem \ref{Th:2.5}.

\bigskip

\section*{Acknowledgments}  The research of Oleksandr Misiats was supported by Simons Collaboration
Grant for Mathematicians No. 854856 and National Science Foundation Grant DMS-2408507. \\

We would also like to thank an anonimous reviewer for suggesting the change of variables idea \eqref{1*}, which lead to a simple yet elegant proof of the long-time behavior of \eqref{1.2*}.

\end{document}